\documentclass[a4paper,11pt]{article}%

\usepackage[subsection]{algorithm}
\usepackage{algorithmic}
\usepackage{amsfonts}
\usepackage{amsmath}
\usepackage{amssymb}
\usepackage{amsthm}
\usepackage{authblk}
\usepackage{bbm}
\usepackage[dvipsnames, usenames]{color}
\usepackage{commath}
\usepackage[inline]{enumitem}
\usepackage{epsfig}  % for figures in .eps.
\usepackage[latin1]{inputenc}
\usepackage{float}
\usepackage{graphicx}
\usepackage{geometry}
\usepackage{leftidx}
\usepackage{mathrsfs}
\usepackage{mathtools}
\usepackage[rightcaption]{sidecap} %for caption next to table
\usepackage{titlesec}
\usepackage{undertilde}
\usepackage{verbatim}
\usepackage[usenames,dvipsnames,svgnames,table]{xcolor}
\usepackage[all]{xy}

\usepackage{hyperref}
\allowdisplaybreaks

\DeclareMathOperator*{\argmax}{arg\,max}

\DeclareMathOperator{\sgn}{sgn}
\newcommand{\defeq}{\vcentcolon=}
\newcommand{\backdefeq}{=\vcentcolon}

\newcommand{\grid}{G}

\newtheorem{assumptions}{Assumptions}[subsection]

\newtheorem{theorem}[assumptions]{Theorem}
\newtheorem{lemma}[assumptions]{Lemma}
\newtheorem{corollary}[assumptions]{Corollary}
\newtheorem{proposition}[assumptions]{Proposition}

\theoremstyle{definition}
\newtheorem{definition}[assumptions]{Definition}
\newtheorem{definitions}[assumptions]{Definitions}

\theoremstyle{remark}
\newtheorem{example}[assumptions]{Example}
\newtheorem{examples}[assumptions]{Examples}
\newtheorem{remark}[assumptions]{Remark}
\newtheorem*{remark*}{Remark}

\newtheorem{notation}[assumptions]{Notation}
\newtheorem*{notation*}{Notation}
\newtheorem{convention}[assumptions]{Convention}
\newtheorem*{convention*}{Convention}

\renewenvironment{abstract}
 {\small
  \begin{center}
  \bfseries \abstractname\vspace{-.05em}\vspace{0pt}
  \end{center}
  \list{}{%
    \setlength{\leftmargin}{.05in}% <---------- CHANGE HERE
    \setlength{\rightmargin}{\leftmargin}%
  }%
  \item\relax}
 {\endlist}

\numberwithin{equation}{subsection}
\allowdisplaybreaks

\title{A fixed-point policy-iteration-type algorithm for symmetric nonzero-sum stochastic impulse control games}
\author{Diego Zabaljauregui\footnote{E-mail: {\tt d.zabaljauregui@lse.ac.uk}}}
\affil{Department of Statistics\\London School of Economics and Political Science} 

\date{ }

\begin{document}
\maketitle

\vspace*{-.9cm}
\begin{abstract}
\noindent Nonzero-sum stochastic differential games with impulse controls offer a realistic and far-reaching modelling framework for applications within finance, energy markets, and other areas, but the difficulty in solving such problems has hindered their proliferation. Semi-analytical approaches make strong assumptions pertaining to very particular cases. To the author's best knowledge, the only numerical method in the literature is the heuristic one we put forward in \cite{ABMZZ} to solve an underlying system of quasi-variational inequalities. Focusing on symmetric games, this paper presents a simpler, more precise and efficient fixed-point policy-iteration-type algorithm which removes the strong dependence on the initial guess and the relaxation scheme of the previous method. A rigorous convergence analysis is undertaken with natural assumptions on the players strategies, which admit graph-theoretic interpretations in the context of weakly chained diagonally dominant matrices. A novel provably convergent single-player impulse control solver is also provided. The main algorithm is used to compute with high precision equilibrium payoffs and Nash equilibria of otherwise very challenging problems, and even some which go beyond the scope of the currently available theory.
\end{abstract}


%%%%%%%%%%%%%%%%%%%%%%%%%%%%%%%%%%%%%%%  INTRODUCTION  %%%%%%%%%%%%%%%%%%%%%%%%%%%%%%%%%%%%%%%%%%%%%%%%%%%%%%%%%%%%
\section*{Introduction}
\addcontentsline{toc}{section}{Introduction}
Stochastic differential games model the interaction between players whose objective functions depend on the evolution of a certain continuous-time stochastic process. The subclass of impulse games focuses on the case where the players only act at discrete (usually random) points in time by shifting the process. In doing so, each of them incurs into costs and possibly generates ``gains" for the others at the same time. They constitute a generalization of the well-known (single-player) optimal impulse control problems \cite[Chpt.7-10]{OS}, which have found a wide range of applications in finance, energy markets and insurance \cite{B,CCTZ,EH,K,ARS}, among plenty of other fields. 

From a deterministic numerical viewpoint, an impulse control problem entails the resolution of a differential quasi-variational inequality (QVI) to compute the value function and, when possible, retrieve an optimal strategy. Policy-iteration-type algorithms \cite{AF,COS,CMS} undoubtedly occupy an ubiquitous place in this respect, especially in the infinite horizon case.   

The presence of a second player makes matters much more challenging, as one needs to find two optimal (or \textit{equilibrium}) payoffs dependent on one another, and the optimal strategies take the form of Nash equilibria (NEs). 
And while impulse controls give a more realistic setting than ``continuous" controls in applications such as the aforementioned, they normally lead to less tractable and technical models.

It is not surprising then, that the literature in impulse games is limited and mainly focused on the zero-sum case \cite{A1,C,EM}. 
The more general and versatile nonzero-sum instance has only recently received attention. The authors of \cite{ABCCV} consider for the first time a general two-player game where both participants act through impulse controls,\footnote{\cite{CWW,WW} also consider nonzero-sum impulse games but assume the intervention times of the players are known from the outset.} and characterize certain type of equilibrium payoffs and NEs via a system of QVIs by means of a verification theorem. Using this result, they provide the first example of an (almost) fully analytically solvable game, motivated by central banks competing over the exchange rate. The result is generalized to $N$ players in \cite{BCG}, which also gives a semi-analytical solution (i.e., depending on several parameters found numerically) to a concrete cash management problem.\footnote{\cite{BCG} also studies the mean field limit game.} A different, more probabilistic, approach is taken in \cite{FK} to find a semi-analytical solution to a strategic pollution control problem and to prove another verification theorem.  

The previous examples, and the lack of others,\footnote{\cite{ABCCV} also gives semi-analytical solutions to modifications of the linear game when changing the payoffs in a non-symmetric way. To the best of the author's knowledge, these are the only examples available at the time of writing.} give testimony of how difficult it is to explicitly solve nonzero-sum impulse games. The analytical approaches require an educated guess to start with and (with the exception of the \textit{linear game} in \cite{ABCCV}) several parameters need to be solved for in general from highly-nonlinear systems of equations coupled with order conditions. All of this can be very difficult, if not prohibitive, when the structure of the game is not simple enough. Further, all of them (as well as the majority of concrete examples in the impulse control literature) assume linear costs. In general, for nonlinear costs, the state to which each player wants to shift the process when intervening is not unique and it depends on the starting point. This effectively means that infinite parameters may need to be solved for, drastically discouraging this methodology. 

While the need for numerical schemes able to handle nonzero-sum impulse games is obvious, unlike in the single-player case, this is an utterly underdeveloped line of research. Focusing on the purely deterministic approach, solving the system of QVIs derived in \cite{ABCCV} involves handling coupled free boundary problems, further complicated by the presence of nonlinear, nonlocal and noncontractive operators. Additionally, solutions will typically be irregular even in the simplest cases such as the linear game. Moreover, the absence of a viscosity solutions framework such as that of impulse control \cite{S} means that it is not possible to know whether the system of QVIs has a solution (not to mention some form of uniqueness) unless one can explicitly solve it. This is further exacerbated by the fact that even defining such a system requires a priori assumptions on the solution (the \textit{unique impulse property}). This is also the case in \cite{FK}.

To the author's best knowledge, the only numerical method available in the literature is our algorithm in \cite{ABMZZ}, which tackles the system of QVIs by sequentially solving single-player impulse control problems combined with a relaxation scheme. Unfortunately, the choice of the relaxation scheme is not obvious in general and the convergence of the algorithm relies on a good initial guess. It was also observed that stagnation could put a cap on the accuracy of the results, without any simple solution to it. Lastly, while numerical validation was performed, no rigorous convergence analysis was provided.

Restricting attention to the one-dimensional infinite horizon two-player case, this paper puts the focus on certain nonzero-sum impulse games which display a symmetric structure between the players. This class is broad enough to include many interesting applications; no less than the competing central banks problem (whether in its linear form \cite{ABCCV} or others considered in the single bank formulation \cite{AF,CZ,J,MO}), the cash management problem \cite{BCG} (reducing its dimension by a simple change of variables) and the generalization of many impulse control problems to the two-player case. 

For this class of games, an iterative algorithm is presented which substantially improves \cite[Alg.2]{ABMZZ} by harnessing the symmetry of the problem, removing the strong dependence on the initial guess and dispensing with the relaxation scheme altogether. The result is a simpler and more intuitive, precise and efficient routine, for which a convergence analysis is provided. It is shown that the overall routine admits a representation that strongly resembles, both algorithmically and in its properties, that of the combined fixed-point policy-iteration methods \cite{HFL1,Cl}, albeit with nonexpansive operators. Still, a certain contraction property can still be established.

To perform the analysis, we impose assumptions on the discretization scheme used on the system of QVIs and the discrete admissible strategies. These naturally generalize those of the impulse control case \cite{AF} and admit graph-theoretic interpretations in terms of weakly chained diagonally dominant (WCDD) matrices and their recently introduced matrix sequences counterpart \cite{A}. We establish a clear parallel between these discrete type assumptions, the behaviour of the players and the Verification Theorem.

Section \ref{s:analytical_problem} deals with the analytical problem. Starting with an overview of the model (Section \ref{s:general_game}), we recall the Verification Theorem of \cite{ABCCV} and the system of QVIs we want to solve (Section \ref{s:QVIs}). We then give a precise definition of the class of symmetric nonzero-sum impulse games and establish some preliminary results (Section \ref{s:symmetric_games}). 

Section \ref{s:discrete_problem} considers the analogous discrete problem. Section \ref{s:dQVIs} specifies a general discrete version of the system of QVIs, such that any discretization scheme compliant with the assumptions to be imposed will enjoy the same properties. Section \ref{s:sym_algo} presents the iterative algorithm, and shows how the impulse control problems that need to be sequentially solved have a unique solution that can be handled by policy iteration. Additionally, Section \ref{s:solve_one_player} provides an alternative general solver for impulse control problems. It consists of an instance of fixed-point policy-iteration that is noncompliant with the standard assumptions \cite{HFL1} and, as far as the author knows, was not used in the context of impulse control before, other than heuristically in \cite{ABMZZ}. We prove its convergence under the present framework.

Section \ref{s:FPPI} characterizes the overall iterative algorithm as a fixed-point policy-iteration-type method, allowing for reformulations of the original problem and results pertaining to the solutions. The necessary matrix and graph-theoretic definitions and results needed are collected in Appendix \ref{appendix:matrices} for the reader's convenience. Section \ref{s:convergence} carries on with the overall convergence analysis and shows to which extent different sets of reasonable assumptions are enough to guarantee convergence to solutions, convergence of strategies and boundedness of iterates. Sufficient conditions for convergence are proved. Discretization schemes are provided in Section \ref{s:discretization}.

Section \ref{s:numerics} presents the numerical results. In Section \ref{s:games_on_fixed_grid}, a variety of symmetric nonzero-sum impulse games, many seemingly too complicated to be handled analytically, are explicitly solved for equilibrium payoffs and NE strategies with great precision. This is done on a fixed grid, while considering different performance metrics and addressing practical matters of implementation. In the absence of a viscosity solutions framework to establish convergence to analytical solutions as the grid is refined, Section \ref{cv_analytical_sol} performs a numerical validation using the only examples of symmetric solvable games in the literature. Section \ref{s:games_without_NE} addresses the case of games without NEs. Section \ref{s:beyond_verif_theo} tackles games beyond the scope of the currently available theory, displaying discontinuous impulses and very irregular payoffs. The latter give insight and motivate further research into this field.

%%%%%%%%%%%%%%%%%%%%%%%%%%%%%%%%%%%%%%  SECTIONS %%%%%%%%%%%%%%%%%%%%%%%%%%%%%%%%%%%%%%%%%%%%%%%%%%%%%%%%%%%%%%%%%%

\section{Analytical continuous-space problem}
\label{s:analytical_problem}
In this section we start by reviewing a general formulation of two-player nonzero-sum stochastic differential games with impulse controls, as considered in \cite{ABCCV}, together with the main theoretical result of the authors: a characterization of certain NEs via a deterministic system of QVIs. 
The indices of the players are denoted $i=1,2$. We will generally use $i$ to indicate a given player and $j$ to indicate their opponent. Since no other type of games is considered in this paper, we will often speak simply of ``games" for brevity. Afterwards we shall specialize the discussion in the (yet to be specified) symmetric instance. 

Throughout the paper, we restrict our attention to the one-dimensional infinite-horizon case. (A similar review is carried out in \cite[Sect.1]{ABMZZ}.) Some of the most technical details concerning the well-posedness of the model are left out for brevity and can be found in \cite[Sect.2]{ABCCV}.

\subsection{General two-player nonzero-sum impulse games}
\label{s:general_game}
Let $(\Omega, \mathcal F, (\mathcal F_t)_{t\geq 0}, \mathbb P)$ be a filtered probability space under the usual conditions supporting a standard one-dimensional Wiener process $W$. We consider two players that observe the evolution of a \textit{state variable} $X$, modifying it when convenient through \textit{controls} of the form $u_i=\{(\tau_i^k,\delta_i^k)\}_{k=1}^\infty$ for $i=1,2$. 
The stopping times $(\tau_i^k)$ are their \textit{intervention times} and the $\mathcal F_{\tau_i^k}$-measurable random variables $(\delta_i^k)$ are their \textit{intervention impulses}. Given controls $(u_1,u_2)$ and a starting point $X_{0^-}=x\in\mathbb R$, we assume $X=X^{x;u_1,u_2}$ has dynamics
\begin{equation}
\label{X}
X_t=x+ \int_0^t\mu(X_s)ds + \int_0^t\sigma(X_s)dW_s + \displaystyle\sum_{k:\ \tau_1^k\leq t}\delta_1^k + \displaystyle\sum_{k:\ \tau_2^k\leq t}\delta_2^k,
\end{equation}
for some given drift and volatility functions $\mu,\sigma:\mathbb R\to \mathbb R$, locally Lipschitz with linear growth.\footnote{See \cite[Def.2.2]{ABCCV} for a precise recursive definition in terms of the strategies.}  

Equation (\ref{X}) states that $X$ evolves as an It\^o diffusion in between the intervention times, and that each intervention consists in shifting $X$ by applying an impulse. It is assumed that the players choose their controls by means of threshold-type \textit{strategies} of the form 
$\varphi_i=(\mathcal I_i,\delta_i)$, where $\mathcal I_i\subseteq\mathbb R$ is a closed set called \textit{intervention (or action) region} and $\delta_i:\mathbb R\to \mathbb R$ is an \textit{impulse function} assumed to be continuous. The complement $\mathcal C_i=\mathcal I_i^c$ is called \textit{continuation (or waiting) region}.
\footnote{In \cite{ABCCV}, strategies are described in terms of continuation regions instead.} 
That is, player $i$ intervenes if and only if the state variable reaches her intervention region, 
by applying an impulse $\delta_i(X_{t^-})$ (or equivalently, shifting $X_{t^-}$ to $X_{t^-}+\delta_i(X_{t^-})$). Further, we impose a priori constraints on the impulses: for each $x\in\mathbb R$ there exists a set $\emptyset\neq\mathcal Z_i(x)\subseteq\mathbb R$ (further specified in Section \ref{s:symmetric_games}) such that $\delta_i(x)\in \mathcal Z_i(x)$ if $x\in\mathcal I_i$.\footnote{In \cite{ABCCV}, $\mathcal Z_i(x)$ is the same for every $x\in\mathbb R$. The generalization in this paper is a standard one in impulse control and will prove useful in the sequel. The results in \cite{ABCCV} still hold with the same proofs.} We also assume the game has no end and player 1 has the priority should they both want to intervene at the same time. (The latter will be excluded later on; see Definition \ref{sym_game} and the remarks that follow it.)

Given a starting point and a pair strategies, the (expected) \textit{payoff} of player $i$ is given by
\begin{equation*}
\label{Ji}
J_i(x;\varphi_1,\varphi_2) \defeq \mathbb E \left[ \int_0^\infty e^{-\rho_i s} f_i(X_s)ds - \sum_{k=1}^\infty e^{-\rho_i \tau_i^k} c_i \big( X_{(\tau_i^k)^-}, \delta_i^k \big) + \sum_{k=1}^\infty e^{-\rho_i \tau_j^k} g_i \big( X_{(\tau_j^k)^-}, \delta_j^k \big)\right],
\end{equation*}
with $X=X^{x;u_1,u_2}=X^{x;\varphi_1,\varphi_2}$. For player $i$, $\rho_i>0$ represents her (subjective) \textit{discount rate}, $f_i:\mathbb R\to\mathbb R$ her \textit{running payoff}, 
$c_i:\mathbb R^2\to(0,+\infty)$ her \textit{cost of intervention} and $g_i:\mathbb R^2\to\mathbb R$ her \textit{gain due to her opponent's intervention} (not necessarily non-negative).
The functions $f_i,c_i,g_i$ are assumed to be continuous.

Throughout the paper, only admissible strategies are considered. Briefly, $(\varphi_1,\varphi_2)$ is \textit{admissible} if it gives well-defined payoffs for all $x\in\mathbb R$, $\|X\|_\infty$ has finite moments and, although each player can intervene immediately after the other, infinite simultaneous interventions are precluded.\footnote{More precisely, these would be $\mathbb R$-admissible strategies. See \cite[Def.2.5]{ABCCV} for more details.} As an example, if the running payoffs have polynomial growth, the ``never intervene strategies"  $\varphi_1=\varphi_2=(\emptyset,\emptyset\hookrightarrow\mathbb R)$ are admissible and the game can be played.

Given a game, we want to know whether it admits some Nash equilibrium and how to compute it. Recall that a pair of strategies $(\varphi_1^*,\varphi_2^*)$ is a \textit{Nash equilibrium} (NE) if for every admissible $(\varphi_1,\varphi_2)$,
$$
J_1(x;\varphi_1^*,\varphi_2^*)\geq J_1(x;\varphi_1,\varphi_2^*)\quad\mbox{and}\quad J_2(x;\varphi_1^*,\varphi_2^*)\geq J_2(x;\varphi_1^*,\varphi_2),
$$
i.e., no player can gain from a unilateral change of strategy. If one such NE exists, we refer to $(V_1,V_2)$, with $V_i(x)= J_i(x;\varphi_1^*,\varphi_2^*)$, as a pair of \textit{equilibrium payoffs}.  

\subsection{General system of quasi-variational inequalities}
\label{s:QVIs}
To present the system of QVIs derived in \cite{ABCCV}, we need to define first the \textit{intervention operators}. For any $V_1, V_2:\mathbb R\to\mathbb R$ and $x\in\mathbb R$, the \textit{loss operator} of player $i$ is defined as 
\begin{equation}
\label{Mi}
\mathcal M_i V_i(x)\defeq\displaystyle\sup_{\delta\in \mathcal Z_i(x)}\{V_i(x+\delta) - c_i(x,\delta)\}.\footnote{Although we could have $\mathcal M_iV_i(x)=+\infty$, this will be excluded when enforcing the system of QVIs (\ref{QVIs}).}
\end{equation}
When applied to an equilibrium payoff, the loss operator $\mathcal M_i$ gives a recomputed present value for player $i$ due to the cost of her own intervention. Given the optimality of the NEs, one would intuitively expect that $\mathcal M_i V_i\leq V_i$ for equilibrium payoffs and that the equality is attained only when it is optimal for player $i$ to intervene. Under this logic:
\begin{definition}
\label{UIP}
We say that the pair $(V_1,V_2)$ has the \textit{unique impulse property (UIP)} if for each $i=1,2$ and $x\in \{\mathcal M_i V_i= V_i\}$, there exists a unique impulse, denoted $\delta_i^*(x)=\delta_i^*(V_i)(x)\in \mathcal Z_i(x)$, that realizes the supremum in (\ref{Mi}).\footnote{We do not require the UIP to hold outside of $\{\mathcal M_i V_i= V_i\}$, as this is not the case for equilibrium payoffs in many examples, such as the linear game with constant costs/gains. Proofs in \cite{ABCCV} carry through unaltered.}
\end{definition}
If $(V_1, V_2)$ enjoys the UIP, we define the \textit{gain operator} of player $i$ as
\begin{equation}
\label{Hi}
\mathcal H_i V_i(x)\defeq V_i(x+\delta_j^*(x)) + g_i(x,\delta_j^*(x)),\quad\mbox{ for }x\in \{\mathcal M_j V_j= V_j\}
\end{equation}
When applied to equilibrium payoffs, the gain operator $\mathcal H_i$ gives a recomputed present value for player $i$ due to her opponent's intervention.

Finally, let us denote by $\mathcal A$ the infinitesimal generator of $X$ when uncontrolled, i.e., 
\begin{equation*}
\mathcal AV(x)\defeq \frac{1}{2}\sigma^2(x) V''(x) + \mu(x)V'(x),
\end{equation*}
for any $V:\mathbb R\to\mathbb R$ which is $C^2$ at some open neighborhood of a given $x\in\mathbb R$. We assume this regularity holds whenever we compute $\mathcal AV(x)$ for some $V$ and $x$. The following Verification Theorem, due to \cite[Thm.3.3]{ABCCV}, states that if a regular enough solution $(V_1,V_2)$ to a certain system of QVIs exists, then it must be a pair of equilibrium payoffs, and a corresponding NE can be retrieved. We state here a simplified version that applies to the one-dimensional infinite-horizon games at hand.\footnote{Unlike in \cite{ABCCV}, there is no terminal condition in the system of QVIs and the assumption that $\partial \mathcal C^*_i$ be a Lipschitz surface is trivially satisfied for an open $\mathcal C^*_i\subseteq\mathbb R$, as it is a countable union of disjoint open intervals.}

\begin{theorem}[\textbf{General system of QVIs}]
\label{verification}
Given a game as in Section \ref{s:general_game}, let $V_1,V_2:\mathbb R\to\mathbb R$ be pair of functions with the UIP, such that for any $i,j \in \{1,2\}$, $i \neq j$:
\begin{equation}
\label{QVIs}
\begin{cases}
	\begin{aligned}
		& \mathcal M_j V_j - V_j \leq 0 && \text{on} \,\,\, \mathbb R   \\
		& \mathcal H_i V_i- V_i=0 && \text{on} \,\,\, \{\mathcal M_j V_j -  V_j = 0\}\backdefeq\mathcal I^*_j  \\
		& \max\big\{\mathcal A V_i -\rho_i V_i + f_i, \mathcal M_i V_i- V_i \}=0 && \text{on} \,\,\, \{\mathcal M_j V_j - V_j < 0\}\backdefeq\mathcal C^*_j
	\end{aligned}
\end{cases}
\end{equation}
and $V_i\in C^2(\mathcal C^*_j\backslash\partial\mathcal C^*_i)\cap C^1(\mathcal C^*_j)\cap C(\mathbb R)$ has polynomial growth and bounded second derivative on some reduced neighbourhood of $\partial\mathcal C^*_i$. Suppose further $\big((\mathcal I^*_i,\delta_i^*)\big)_{i=1,2}$ are admissible strategies.\footnote{For consistency with the strategies' definition, one should assume that $\delta^*$ has been continuously extended to $\mathbb R$. The conclusion is unaffected by the choice of the extension.} 
$$\mbox{Then,}\quad
(V_1,V_2)\mbox{ are equilibrium payoffs attained at a NE }\big((\mathcal I^*_i,\delta_i^*)\big)_{i=1,2}.$$
\end{theorem}
The first equation of system (\ref{QVIs}) states that at an equilibrium, a player cannot increase her own payoff by a unilateral intervention. One therefore expects that the equality $\mathcal M_j V_j = V_j$ will only hold when player $j$ intervenes, or in other words, when the value she gains can compensate the cost of her intervention. Consequently, the second equation says that a gain results from the opponent's intervention. Finally, the last one, means that when the opponent does not intervene, each player faces a single-player impulse control problem. 

We conclude this section with some final observations that will be relevant in the sequel: 
\begin{remark}
\label{r:zero_impulse}
An immediate consequence of assuming strictly positive costs is that intervening at any state with a null impulse reduces the payoff of the acting player and is therefore suboptimal. This is also displayed in system (\ref{QVIs}): if at some state $x$, $\mathcal M_jV_j(x)$ was realized for $\delta=0$, then $\mathcal M_jV_j(x)= V_j(x) + c_j(x,0)<V_j(x)$. At the same time, allowing for vanishing costs often leads to degenerate games in the current framework \cite[Sect.4.4]{ABCCV}. Hence, assuming $c_j>0$ is quite reasonable.
\end{remark}

\begin{remark}
\label{r:concave_costs}
Consider the case of nonegative impulses and cost functions being strictly concave in the impulse as in \cite{C, EM}. That is, $c_i(x,\delta+\bar\delta)<c_i(x,\delta)+c_i(x+\delta,\bar\delta)$ for all $x\in\mathbb R,\ \delta,\bar\delta\geq 0$. This models the situation in which simultaneous interventions are more expensive than a single one to the same effect. In such cases, it is easy to see that in the context of Theorem \ref{verification}, player $i$ will only shift the state variable towards her continuation region.\footnote{Let $x\in\mathcal I^*_i$ and suppose $y_i^*\defeq x+\delta^*_i(x)\in\mathcal I^*_i$. Set $y_i^{**}\defeq y_i^* + \delta^*_i(y_i^*)$. Then, by the UIP, the definitions of $\delta^*_i(x)$  and $\mathcal I^*_i$, and the concavity of the cost: 
$V_i(y_i^{**})-c_i(x,\delta^*_i(x)+\delta^*_i(y_i^*))\leq V_i(y_i^*)-c_i(x,\delta^*_i(x))=V_i(y_i^{**})-c_i(y_i^*,\delta^*_i(y_i^*))-c_i(x,\delta^*_i(x))<V_i(y_i^{**})-c_i(x,\delta^*_i(x)+\delta^*_i(y_i^*))$
, which is a contradiction.}
\end{remark}

\subsection{Symmetric two-player nonzero-sum impulse games}
\label{s:symmetric_games}
We want to focus our study on games which present a certain type of symmetric structure between the players, generalising the linear game \cite{ABCCV} and the cash management game \cite{BCG}.\footnote{The latter can be reduced to one dimension with the change of variables $x=x_1-x_2$. Additionally, we will restrict attention to unidirectional impulses, as these yield the ``most relevant" NE found in \cite{BCG}.}

\begin{notation*}
The type of games presented in Section \ref{s:general_game} are fully defined by setting the drift, volatility, impulse constraints, discount rates, running payoffs, costs and gains. In other words, any such game can be represented by a tuple $\mathcal G=(\mu,\sigma,\mathcal Z_i,\rho_i,f_i,c_i,g_i)_{i=1,2}$.
\end{notation*} 

\begin{definitions}
\label{sym_game}
We say that a game $\mathcal G=(\mu,\sigma,\mathcal Z_i,\rho_i,f_i,c_i,g_i)_{i=1,2}$ is \textit{symmetric (with respect to zero)} if
	\begin{enumerate}[label=(S\arabic*)]
		\item \label{sym_dynamics} $\mu$ is odd and $\sigma$ is even (i.e., $\mu(x)=-\mu(-x)$ and $\sigma(x)=\sigma(-x)$ for all $x\in\mathbb R$).
		\item \label{sym_constraints} $-\mathcal Z_2(-x)=\mathcal Z_1(x)\subseteq [0,+\infty)$ for all $x\in\mathbb R$ and $\mathcal Z_1(x)=\{0\}=\mathcal Z_2(-x)$ for all $x\geq 0$.
		\item \label{sym_rest} $\rho_1=\rho_2$, $f_1(x)=f_2(-x)$, $c_1(x,\delta)=c_2(-x,-\delta)$ and $g_1(x,-\delta)=g_2(-x,\delta)$, for all $\delta\in\mathcal Z_1(x),\ x\in\mathbb R$.
	\end{enumerate}
We say that the game is \textit{symmetric with respect to }$s$ (for some $s\in\mathbb R$), if the $s$-\textit{shifted} game $(\mu(x+s),\sigma(x+s),\mathcal Z_i(x+s),\rho_i,f_i(x+s),c_i(x+s,\delta),g_i(x+s,\delta))_{i=1,2}$ is symmetric. We refer to $x=s$ as a \textit{symmetry line} of the game. 
\end{definitions}
Condition \ref{sym_dynamics} is necessary for the state variable to have symmetric dynamics. In particular, together with \ref{sym_rest}, it guarantees symmetry between solutions of the Hamilton--Jacobi--Bellman (HJB) equations of the players when there are no interventions, i.e.,
$$
V=V^*(x)\mbox{ solves }\mathcal AV -\rho_1V +f_1=0\quad\mbox{ if and only if }\quad V=V^*(-x)\mbox{ solves }\mathcal AV -\rho_2V +f_2=0.
$$
\begin{examples}
The most common examples of It\^o diffusions satisfying this assumption are the scaled Brownian motion (symmetric with respect to zero) and the Ornstein--Uhlenbeck (OU) process (symmetric with respect to its long term mean). 
\end{examples}

Condition \ref{sym_rest} is self-explanatory, while \ref{sym_constraints} is only partly so. Indeed, although symmetric constraints on the impulses $\mathcal Z_1(x)=-\mathcal Z_2(-x)$ should clearly be a requirement, the rest of \textit{(ii)} is in fact motivated by the numerical method to be presented and the type of problems it can handle. On the one hand, the third equation of the QVIs system (\ref{QVIs}) implies that a stochastic impulse control problem for player $i$ needs to be solved on $\mathcal C^*_j$. The unidirectional impulses assumption is a common one for the convergence of policy iteration algorithms in impulse control.\footnote{See this assumption in \cite[Sect.4]{CMS} or \cite[Sect.10.4.2]{OS}, its graph-theoretic counterpart in \cite[Asm.(H2) and Thm.4.3]{AF}, and a counterexample of convergence in its absence in \cite[Ex.4.9]{AF}.} However, it is often too restrictive for many interesting applications,\footnote{See \cite{RZ, RZ1} for a novel, globally convergent, combined policy iteration and penalization scheme for general impulse control HJBQVIs.} such as when the controller would benefit the most from keeping the state variable within some bounded interval instead of simply keeping it ``high" or ``low" (see, e.g., \cite{B} and \cite[Sect.6.1]{AF}). Interestingly enough, assuming unidirectional impulses turns out to be less restrictive when there is a second player present, with an opposed objective. Indeed, it can happen that each player needs not to intervene in one of the two directions, and can instead rely on her opponent doing so, while capitalising a gain rather than paying a cost. See examples in Section \ref{s:games_on_fixed_grid} with quadratic and degree four running payoffs. 

On the other hand, $\mathcal Z_1(x)=\{0\}=\mathcal Z_2(-x)$ for all $x\geq 0$ means that we can assume without loss of generality that the admissible intervention regions do not cross over the symmetry line; i.e., $\mathcal I_1\subseteq(-\infty,0)$ and $\mathcal I_2\subseteq (0,+\infty)$ for every pair of strategies. (See Remark \ref{r:zero_impulse}.) This guarantees in particular that the players never want to intervene at the same time and the priority rule can be disregarded.

There are different reasons why the last mentioned condition is less restrictive than it first appears to be. It is not uncommon to assume connectedness of either intervention or continuation regions (or other conditions implying them) both in impulse control \cite{E} and nonzero-sum games \cite[Sect.1.2.1]{DFM}. 
The same can be said for assumptions that prevent the players from intervening in unison \cite[Sect.1.2.1]{DFM},\cite[Rmk.6.5]{C}.
\footnote{In \cite[Rmk.6.5]{C}, assumptions are given to guarantee the zero-sum analogous to $\mathcal M_iV_i\leq \mathcal H_iV_i$, with a strict inequality if such assumptions are slightly strengthened. These inequalities, in the context of Theorem \ref{verification}, imply that the equilibrium intervention regions cannot intersect.} In the context of symmetric games and payoffs (see Lemma \ref{sym_lemma}) such assumptions would necessarily imply the intervention regions need to be on opposed sides of the symmetry line. Additionally, without any further requirements, strategies such that $\mathcal I_1\supseteq(-\infty,0]$ and $\mathcal I_2\supseteq [0,+\infty)$ would be inadmissible in the present framework, as per yielding infinite simultaneous impulses.

\begin{definitions}
\label{sym_strategies}
Given a symmetric game, we say that $\big((\mathcal I_i,\delta_i)\big)_{i=1,2}$ are \textit{symmetric strategies (with respect to zero)} if $\mathcal I_1=-\mathcal I_2$ and $\delta_1(x)=-\delta_2(-x)$. 
Given a symmetric game with respect to some $s\in\mathbb R$, we say that$\big((\mathcal I_i,\delta_i)\big)_{i=1,2}$ are \textit{symmetric strategies with respect to $s$} if $\big((\mathcal I_i-s,\delta_i(x+s))\big)_{i=1,2}$ are symmetric, and we refer to $x=s$ as a \textit{symmetry line} of the strategies. 
\end{definitions}

\begin{definition}
\label{sym_functions}
We say that $V_1,V_2:\mathbb R\to\mathbb R$ are \textit{symmetric functions (with respect to zero)} if $V_1(x)=V_2(-x)$. We say that they are \textit{symmetric functions with respect to $s$} (for some $s\in\mathbb R$) if $V_1(x+s),V_2(x+s)$ are symmetric, and we refer to $x=s$ as a \textit{symmetry line} for $V_1,V_2$. 
\end{definition}

\begin{remark}
\label{r:sym_NE}
Definition \ref{sym_strategies} singles out strategies that share the same symmetry line with the game. For the linear game, for example, the authors find infinitely many NEs \cite[Prop.4.7]{ABCCV}, each presenting symmetry with respect to some point $s$, but only one for $s=0$ (hence, symmetric in the sense of Definition \ref{sym_strategies}). At the same time, the latter is the only one for which the corresponding equilibrium payoffs $V_1,V_2$ have a symmetry line as per Definition \ref{sym_functions}. The same is true for the cash management game \cite{BCG}.
\end{remark}

\begin{remark}
Throughout the paper we will work only with games symmetric with respect to zero, to simplify the notation. Working with any other symmetry line amounts simply to shifting the game and results back and forth.
\end{remark}

\begin{lemma}
\label{sym_lemma}
For any symmetric game, strategies $(\varphi_1,\varphi_2)$ and functions $V_1,V_2:\mathbb R\to\mathbb R$:
\begin{enumerate}[label=(\roman*)]
\item If $V_1,V_2$ are symmetric, then $\mathcal M_1V_1,\mathcal M_2V_2$ are symmetric.
\item If $V_1,V_2$ are symmetric and have the UIP, then $\delta_1^*(x)=-\delta_2^*(-x)$ and $\mathcal H_1V_1,\mathcal H_2V_2$ are symmetric. 
\item If $(\varphi_1,\varphi_2)$ are symmetric, 
then $J_1(\cdot;\varphi_1,\varphi_2),J_2(\cdot;\varphi_1,\varphi_2)$ are symmetric. 
\item If $V_1,V_2$ are as in Theorem \ref{verification} and $(\varphi^*_1,\varphi^*_2)$ is the corresponding NE of the theorem, then $(\varphi^*_1,\varphi^*_2)$ are symmetric if and only if $V_1,V_2$ are symmetric. 
\end{enumerate}
\end{lemma}

\begin{proof}
\textit{(i)} and \textit{(ii)} are straightforward from the definitions. 

To see \textit{(iii)}, one can check with the recursive definition of the state variable \cite[Def.2.2]{ABCCV} that $X^{-x;\varphi_1,\varphi_2}$ has the same law as $-X^{x;\varphi_1,\varphi_2}$ (recall that the continuation regions are simply disjoint unions of open intervals). Noting also that intervention times and impulses are nothing but jump times and sizes of $X$, one concludes that $J_1(x;\varphi_1,\varphi_2)=J_2(-x;\varphi_1,\varphi_2)$, as intended.  

Finally, \textit{(iv)} is a consequence of \textit{(i)}, \textit{(ii)} and \textit{(iii)}. 
\end{proof}

\begin{convention}
\label{convention}
In light of Lemma \ref{sym_lemma}, for any symmetric game we will often lose the player index from the notations and refer always to 
quantities corresponding to player 1,\footnote{Note that $g$ will denote $g(x,\delta)\defeq g_1(x,-\delta)$, as $\delta\geq 0$ for player 1, yet $g_1$ depends on the (negative) impulse of player 2.} henceforth addressed simply as ``the player". Player 2 shall be referred to as ``the opponent". Statements like ``$V$ has the UIP" or ``$V$ is a symmetric equilibrium payoff" are understood to refer to $(V(x),V(-x))$. Likewise, ``$(\mathcal I,\delta)$ is admissible" or ``$(\mathcal I,\delta)$ is a NE" refer to the pair $(\mathcal I,\delta(x)),(-\mathcal I,-\delta(-x))$.
\end{convention}

Due to their general lack of uniqueness, it is customary in game theory to restrict attention to specific type of NEs, depending on the problem at hand (see for instance \cite{HS} for a treatment within the classical theory). Motivated by Lemma \ref{sym_lemma} \textit{(iii)} and \textit{(iv)}, and by Remark \ref{r:sym_NE}, one can arguably state that symmetric NEs are the most meaningful for symmetric games. Furthermore, Lemma \ref{sym_lemma} implies that for symmetric games, one can considerably reduce the complexity of the full system of QVIs (\ref{QVIs}) provided the conjectured NE (or equivalently, the pair of payoffs) is symmetric. Using Convention \ref{convention}, Theorem \ref{verification} and Lemma \ref{sym_lemma} give:
\begin{corollary}[\textbf{Symmetric system of QVIs}]
\label{coro_sym_QVIs}
Given a symmetric game as in Definition \ref{sym_game}, let $V:\mathbb R\to\mathbb R$ be a function with the UIP, such that:
\begin{equation}
\label{sym_QVIs}
\begin{cases}
	\begin{aligned}
		& \mathcal H V- V=0 && \text{on} \,\,\, -\{\mathcal M V -  V = 0\}\backdefeq -\mathcal I^*\\
		& \max\big\{\mathcal A V -\rho V + f, \mathcal M V- V \}=0 && \text{on} \,\,\, -\{\mathcal M V - V < 0\}\backdefeq-\mathcal C^*
	\end{aligned}
\end{cases}
\end{equation} 
and $V\in C^2(-\mathcal C^*\backslash\partial\mathcal C^*)\cap C^1(-\mathcal C^*)\cap C(\mathbb R)$ has polynomial growth and bounded second derivative on some reduced neighbourhood of $\partial\mathcal C^*$. Suppose further that $(\mathcal I^*,\delta^*)$ is an admissible strategy.
$$\mbox{Then,}\quad V\mbox{ is a symmetric equilibrium payoff attained at a symmetric NE }(\mathcal I^*,\delta^*).$$
\end{corollary}
Note that system (\ref{sym_QVIs}) also omits the equation $\mathcal M V - V \leq 0$ as per being redundant. Indeed, by Definition \ref{sym_game} and Remark \ref{r:zero_impulse}, at a NE the player does not intervene above 0, nor the opponent below it. Thus, $\mathcal M V - V \leq \max\big\{\mathcal A V -\rho V + f, \mathcal M V- V \}=0$ on $-\mathcal C^*\supset (-\infty,0]$ and $\mathcal M V- V < 0$ on $[0,+\infty)$. 

System (\ref{sym_QVIs}) simplifies a numerical problem which is very challenging even in cases of linear structure \cite{ABMZZ}. In light of the previous, we will focus our attention on symmetric NEs only and numerically solving the reduced system of QVIs (\ref{sym_QVIs}).

\section{Numerical discrete-space problem}
\label{s:discrete_problem}
In this section we consider a discrete version of the symmetric system of QVIs (\ref{sym_QVIs}) over a fixed grid, and propose and study an iterative method to solve it. As it is often done in numerical analysis for stochastic control, for the sake of generality we proceed first in an abstract fashion without making reference to any particular discretization scheme. Instead, we give some general assumptions any such scheme should satisfy for the results presented to hold. Explicit discretization schemes within our framework are presented in Section \ref{s:discretization} and used in Section \ref{s:numerics}.  

\subsection{Discrete system of quasi-variational inequalities}
\label{s:dQVIs}

From now on we work on a discrete symmetric grid 
\begin{equation*}
\grid:\ x_{-N}=-x_N<\dots<x_{-1}=-x_1<x_0=0<x_1<\dots<x_N.
\end{equation*} 
$\mathbb R^\grid$ denotes the set of functions $v:\grid\to\mathbb R$ and $S:\mathbb R^\grid\to \mathbb R^\grid$ denotes the symmetry operator, $Sv(x)=v(-x)$. In general, by an ``operator" we simply mean some $F:\mathbb R^\grid\to\mathbb R^\grid$, not necessarily linear nor affine unless explicitly stated. We shall identify grid points with indices, functions in $\mathbb R^\grid$ with vectors and linear operators with matrices; e.g., $S=(S_{ij})$ with $S_{ij}=1$ if $x_i=-x_j$ and 0 otherwise. The (partial) order considered in $\mathbb R^\grid$ and $\mathbb R^{\grid\times\grid}$ is the usual pointwise order for functions (elementwise for vectors and matrices), and the same is true for the supremum, maximum and arg-maximum induced by it.

We want to solve the following discrete nonlinear system of QVIs for $v\in\mathbb R^\grid$:
\begin{equation}
\label{dQVIs}
\begin{cases}
	\begin{aligned}
		& H v- v=0 && \text{on} \,\,\, -\{ M v -  v = 0\}\backdefeq -I^*\\
		& \max\big\{ Lv + f, Mv- v \}=0 && \text{on} \,\,\, -\{ M v - v < 0\}\backdefeq -C^*,
	\end{aligned}
\end{cases}
\end{equation} 
where $f\in\mathbb R^\grid$ and $L:\mathbb R^\grid\to \mathbb R^\grid$ is a linear operator. The nonlinear operators $M,H:\mathbb R^\grid\to\mathbb R^\grid$ are as follows: let $\emptyset\neq Z(x)\subseteq\mathbb R$ be a finite set for each $x\in\grid$, with $Z(x)=\{0\}$ if $x\geq 0$. Set $Z\defeq\prod_{x\in\mathbb\grid}Z(x)$ and for each $\delta\in Z$ let $B(\delta):\mathbb R^\grid\to\mathbb R^\grid$ be a linear operator, $c(\delta)\in\mathbb (0,+\infty)^\grid$ and $g(\delta)\in\mathbb R^\grid$, the three of them being \textit{row-decoupled} in the sense of \cite{BMZ, AF} (i.e., row $x$ of $B(\delta),c(\delta),g(\delta)$ depends only on $\delta(x)\in Z(x)$). Then 
\begin{gather}
Mv \defeq \max_{\delta\in Z}\big\{B(\delta)v-c(\delta)\big\},\quad Hv=H(\delta^*)v\defeq SB(\delta^*)Sv+g(S\delta^*)\label{discrete_intervention_operators}\\
\mbox{and}\quad\delta^*=\delta^*(v)\defeq\max\Big(\argmax_{\delta\in Z}\big\{B(\delta)v-c(\delta)\big\}\Big).\label{delta_star}
\end{gather}

Some remarks are in order. Firstly, in the same fashion as the continuous-space case, the sets $I^*,C^*$ form a partition of the grid and represent the (discrete) intervention and continuation regions of the player, while $-I^*,-C^*$ are such regions for the opponent. 

Secondly, the general representation of $M$ follows \cite{CMS,AF}. For the standard choices of $B(\delta)$, our definition of $H$ is the only one for which a discrete version of Lemma \ref{sym_lemma} holds true (see Section \ref{s:discretization}). However, since $B$ and $g$ are row-decoupled, $SB(\delta^*)S$ and $g(S\delta^*)$ cannot be, as each row $x$ depends on $\delta^*(-x)$. For this reason and the lack of maximization over $-I^*$, there is no obvious way to reduce problem (\ref{dQVIs}) to a classical Bellman problem:
\begin{equation}
\label{Bellman_problem}
\sup_{\varphi}\big\{-A(\varphi)v+b(\varphi)\big\}=0,
\end{equation}
like in the impulse control case \cite{AF}, to apply Howard's policy iteration \cite[Ho-1]{BMZ}. Furthermore, unlike in the control case, even with unidirectional impulses and good properties for $L$ and $B(\delta)$, system (\ref{dQVIs}) may have no solution as in the analytical case \cite{ABCCV}. 

Thirdly, we have defined $\delta^*$ in (\ref{delta_star}) by choosing one particular maximizing impulse for each $x\in\grid$. The main motivation behind fixing one is to have a well defined discrete system of QVIs for every $v\in\mathbb R^\grid$. (This is not the case for the analytical problem (\ref{sym_QVIs}) where the gain operator $\mathcal H$ is not well defined unless $V$ has the UIP.) Being able to plug in any $v$ in (\ref{dQVIs}) and obtain a residual will be useful in practice, when assessing the convergence of the algorithm (see Section \ref{s:numerics}). Whether a numerical solution verifies, at least approximately, a discrete UIP (and the remaining technical conditions of the Verification Theorem) becomes something to be checked separately a posteriori. 

\begin{remark}
\label{r:maxargmax1}
Choosing the maximum arg-maximum in (\ref{delta_star}) is partly motivated by ensuring a discrete solution will inherit the property of Remark \ref{r:concave_costs}. (The proof remains the same, for the discretizations of Section \ref{s:discretization}.) We will also motivate it in terms of the proposed numerical algorithm in Remark \ref{r:maxargmax2}. Note that in \cite{ABMZZ} the minimum arg-maximum is used instead for both players. Nevertheless, the replication of property \textit{(ii)}, Lemma \ref{sym_lemma}, dictates that it is only possible to be consistent with \cite{ABMZZ} for one of the two players (in this case, the opponent).
\end{remark}

\subsection{Iterative algorithm for symmetric games}
\label{s:sym_algo}
This section introduces the iterative algorithm developed to solve system (\ref{dQVIs}), which builds on \cite[Alg.2]{ABMZZ} by harnessing the symmetry of the problem and dispenses with the need for a relaxation scheme altogether. It is presented with a pseudocode that highlights the mimicking of system (\ref{dQVIs}) and the intuition behind the algorithm; namely:
\begin{itemize}
\item The player starts with some suboptimal strategy $\varphi^0=(I^0,\delta^0)$ and payoff $v^0$, to which the opponent responds symmetrically, resulting in a gain for the player (first equation of (\ref{dQVIs}); lines 1, 2 and 4 of Algorithm \ref{sym_algo}).
\item The player improves her strategy by choosing the optimal response, i.e., by solving a single-player impulse control problem through a policy-iteration-type algorithm (second equation of ({\ref{dQVIs}}); line 5 of Algorithm \ref{sym_algo}).
\item This procedure is iterated until reaching a stopping criteria (lines 6-8 of Algorithm \ref{sym_algo}).
\end{itemize}
\begin{notation}
In the following: $\grid_{<0}$ and $\grid_{\leq 0}$ represent the sets of grid points which are negative and nonpositive respectively, and $\Phi$ the set of \textit{(discrete) strategies} 
\begin{equation}
\label{discrete_strategies}
\Phi\defeq\{\varphi=(I,\delta):\ I\subseteq\grid_{<0}\mbox{ and }\delta\in Z\}.
\end{equation}
Set complements are taken with respect to the whole grid, $\operatorname{Id}:\mathbb R^G\to\mathbb R^G$ is the identity operator; and given a linear operator $O:\mathbb R^\grid\to\mathbb R^\grid\simeq\mathbb R^{\grid\times\grid}$, $v\in\mathbb R^\grid$ and subsets $I,J\subseteq\grid$, $v_I\in\mathbb R^I$ denotes the restriction of $v$ to $I$ and $O_{IJ}\in\mathbb R^{I\times J}$ the submatrix/operator with rows in $I$ and columns in $J$. 
\end{notation}

\begin{algorithm}[H]
    \caption{Iterative algorithm for symmetric games} \label{sym_algo}
		Set $tol,scale>0$. 
		\begin{algorithmic}[1]
	    \STATE Choose initial guess: $v^0\in\mathbb R^\grid$
			\STATE Set $I^0=\{Lv^0+f\leq Mv^0-v^0\}\cap\grid_{<0}$ and $\delta^0=\delta^*(v^0)$
			\FOR{$k=0,1,\dots$}
				\STATE $
									v^{k+1/2}=\begin{cases} 
														H(\delta^k)v^k & \mbox{ on }-I^k\\
													  v^k            & \mbox{ on }(-I^k)^c
														\end{cases}
								$						
				\STATE $(v^{k+1},I^{k+1},\delta^{k+1})=\textsc{SolveImpulseControl}(v^{k+1/2},(-I^k)^c)$
				\IF{$\|(v^{k+1}-v^k)/\max\{|v^{k+1}|,scale\}\|_{\infty}<tol$} 
					\STATE break from the iteration
				\ENDIF	
			\ENDFOR	
    \end{algorithmic}
\end{algorithm}

The $scale$ parameter in line 5 of Algorithm \ref{sym_algo}, used throughout the literature by Forsyth, Labahn and coauthors \cite{AF,FL,HFL1,HFL2}, prevents the enforcement of unrealistic levels of accuracy for points $x$ where $v^{k+1}(x)\approx 0$. Additionally, note that having chosen the initial guess for the payoff $v^0$, the initial guess for the strategy is induced by $v^0$. (The alternative expression for the intervention region gives the same as $\big\{M v^0- v^0=0\big\}$ for a solution of (\ref{dQVIs}).)

Line 5 of Algorithm \ref{sym_algo} assumes we have a subroutine $\textsc{SolveImpulseControl}(w,D)$ that solves the constrained QVI problem:
\begin{equation}
\label{constrained_QVI}
\max\{Lv+f,Mv-v\}=0\mbox{ on }D,\quad\mbox{subject to }v=w\mbox{ on }D^c
\end{equation}
for fixed $\grid_{\leq 0}\subseteq D\subseteq\grid$ (approximate continuation region of the opponent) and $w\in\mathbb R^\grid$ (approximate payoff due to the opponent's intervention). Although we only need to solve for $\tilde v=v_D$, the value of $v_{D^c}=w_{D^c}$ impacts the solution both when restricting the equations and when applying the nonlocal operator $M$. Hence, the approximate payoff $v^{k+1/2}$ fed to the subroutine serves to pass on the gain that resulted from the opponent's intervention and as an initial guess if desired (more on this on Remark \ref{r:initial_guess}).

The remaining of this section consists in establishing an equivalence between problem (\ref{constrained_QVI}) and a classical (unconstrained) QVI problem of impulse control. This allows us to prove the existence and uniqueness of its solution. In particular, we will see that \textsc{SolveImpulseControl} can be defined, if wanted, by policy iteration (see the next section for an alternative method and remarks on other possible ones). Let us suppose from here onwards that the following assumptions hold true (see Appendix \ref{appendix:matrices} for the relevant Definitions \ref{matrix_definitions}):
\begin{enumerate}[label=(A\arabic*), start=0]
\item \label{A0} 
For each strategy $\varphi=(I,\delta)\in\Phi$ and $x\in I$, there exists a walk in graph$B(\delta)$ from row $x$ to some row $y\in I^c$. 
\item \label{A1} 
$-L$ is a strictly diagonally dominant (SDD) $\mbox{L}_0$-matrix and, for each $\delta\in Z$, $\operatorname{Id}-B(\delta)$ is a weakly diagonally dominant (WDD) $\mbox{L}_0$-matrix.
\end{enumerate}

\begin{remark}[\textit{Interpretation}]
\label{r:interpretation}
Assumptions \ref{A0},\ref{A1} are (H2),(H3) in \cite{AF}. For an impulse operator (say, ``$B(\delta)v(x)=v(x+\delta)$"), \ref{A0} asserts that the player always wants to shift states in her intervention region to her continuation region through finitely many impulses. (This does not take into account the opponent's response.) On the other hand, if problem (\ref{constrained_QVI}) was rewritten as a fixed point problem, \ref{A1} would essentially mean that the uncontrolled operator is contractive while the controlled ones are nonexpansive (see \cite{CMS} and \cite[Sect.4]{AF}). 
\end{remark}

\begin{theorem}
\label{solution_constrained_QVI_1}
Assume \emph{\ref{A0}},\emph{\ref{A1}}. Then, for every $\grid_{\leq 0}\subseteq D\subseteq\grid$ and $w\in\mathbb R^\grid$, there exists a unique $v^*\in\mathbb R^\grid$ that solves the constrained QVI problem (\ref{constrained_QVI}). Further, $v^*_D$ is the unique solution of 
\begin{equation}
\label{restricted_QVI}
\max\Big\{\tilde L \tilde v +\tilde f, \tilde M \tilde v - \tilde v\Big\}=0,
\end{equation}
where $\tilde L\defeq L_{DD},\ \tilde f\defeq f_D+L_{DD^c}w_{D^c},\ \tilde Z\defeq \prod_{x\in D}Z(x),\ \tilde c(\tilde\delta)\defeq c(\tilde\delta)-B(\tilde\delta)_{DD^c}w_{D^c}\mbox{ and }\tilde B(\tilde\delta)\defeq B(\tilde\delta)_{DD}\mbox{ for }\tilde\delta\in\tilde Z;\mbox{ and } \tilde M\tilde v\defeq\max_{\tilde\delta\in\tilde Z}\big\{\tilde B(\tilde\delta)\tilde v -\tilde c(\tilde\delta)\big\}\mbox{ for }\tilde v\in\mathbb R^D$.

Additionally, for any initial guess, the sequence $(\tilde v^k)\subseteq\mathbb R^D$ defined by policy iteration \cite[Thm.4.3]{AF} applied to problem (\ref{restricted_QVI}), converges exactly to $v^*_D$ in at most $|\tilde\Phi|$ iterations, with $\tilde\Phi\defeq\{\tilde\varphi=(I,\tilde\delta):\ I\subseteq\grid_{<0}\mbox{ and }\tilde\delta\in\tilde Z\}$  the set of restricted admissible strategies.\footnote{$|A|$ denotes the cardinal of set $A$.}
\end{theorem}

\begin{proof}
The equivalence between problems (\ref{constrained_QVI}) and (\ref{restricted_QVI}) is due to simple algebraic manipulation and $B(\delta),c(\delta)$ being row-decoupled for every $\delta\in Z$. $B(\tilde\delta),c(\tilde\delta)$ are defined in the obvious way for each $\tilde\delta\in\tilde Z$.

The rest of the proof is mostly as in \cite[Thm.4.3]{AF}. Let $\tilde{\operatorname{Id}}=\operatorname{Id}_{DD}$. Each intervention region $I$ can be identified with its indicator $\tilde\psi=\mathbbm 1_I\in\{0,1\}^D$ since $D\supseteq I$. In turn, each $\tilde\psi$ can be identified with the diagonal matrix $\tilde\Psi=\mbox{diag}(\tilde\psi)\in\mathbb R^{D\times D}$. Then, problem (\ref{restricted_QVI}) takes the form of the classical Bellman problem 
\begin{equation}
\label{equivalent_Bellman}
\max_{\tilde\varphi\in\tilde\Phi}\big\{-A(\tilde\varphi)v+b(\tilde\varphi)\big\}=0,
\end{equation}
if we take 
$$A(\tilde\varphi)=-(\tilde{\operatorname{Id}}-\tilde\Psi)\tilde L + \tilde\Psi(\tilde{\operatorname{Id}}-\tilde B(\tilde\delta))\quad\mbox{ and }\quad b(\tilde\varphi)=(\tilde{\operatorname{Id}}-\tilde\Psi)\tilde f - \tilde\Psi\tilde c.$$ 
Note that $\tilde\Phi$ can be identified with the Cartesian product 
$$\tilde\Phi=\Big(\prod_{x\in\grid_{<0}}\{0,1\}\times Z(x)\Big)\times\Big(\prod_{x\in D\backslash\grid_{<0}}\{0\}\times Z(x)\Big)$$
and $A(\tilde\varphi),b(\tilde\varphi)$ are row-decoupled for every $\tilde\varphi\in\tilde\Phi$. Since $\tilde\Phi$ is finite, all we need to show is that the matrices $A(\tilde\varphi)$ are monotone (Definitions \ref{matrix_definitions} and \cite[Thm.2.1]{BMZ}). Let us check the stronger property (Thm. \ref{characterization_theorem} and Prop. \ref{M-matrix_characterization}) of being weakly chained diagonally dominant (WCDD) $\mbox{L}_0$-matrices (see Definitions \ref{matrix_definitions}).

If \ref{A0} and \ref{A1} also held true for the restricted matrices and strategies, the conclusion would follow. While \ref{A1} is clearly inherited, \ref{A0} may fail to do so, but only in non-problematic cases. To see this, let $\tilde\varphi=(I,\tilde\delta)\in\tilde\Phi,\ x\in I\subseteq D$ and let $\delta\in Z$ be some extension of $\tilde\delta$. Note that row $x$ of $A(\tilde\varphi)$ is WDD. We want to show that there is a walk in graph$A(\tilde\varphi)$ from $x$ to an SDD row.

By \ref{A0} there must exist some walk $x=y_0\to y_1\to\dots\to y_n\in I^c$ in graph$B(\delta)$. If this is in fact a walk from $x$ to $I^c\cap D$ in graph$\tilde B(\tilde\delta)$, then it verifies the desired property (just as in \cite[Thm.4.3]{AF}). If not, then there must be a first $0\leq m< n$ such that the subwalk $x\to \dots\to y_m$ is in graph$\tilde B(\tilde\delta)$ but $y_{m+1}\notin D$. Since $y_m\to y_{m+1}$ is an edge in graph$B(\delta)$, we have $B(\delta)_{y_m,y_{m+1}}\neq 0$ and the WDD row (by \ref{A1}) $y_m$ of $\tilde{\operatorname{Id}}-\tilde B(\tilde\delta)$ is in fact SDD. Meaning that the subwalk $x\to \dots\to y_m$ verifies the desired property instead.
\end{proof}

\begin{remark}(\textit{Practical considerations}) 
\label{practical_considerations}
\begin{enumerate*}
\item While convergence is guaranteed to be exact, floating point arithmetic can bring about stagnation \cite{HFL2}. A stopping criteria like that of Algorithm \ref{sym_algo} should be used in those cases, with a tolerance $\ll tol$.   
\item \label{lambda} The solution of system (\ref{restricted_QVI}) does not change if one introduces a scaling factor $\lambda>0$: $\max\big\{\tilde L \tilde v +\tilde f, \lambda\big(\tilde M \tilde v - \tilde v\big)\big\}=0$ \cite[Lem.4.1]{AF}. This problem-specific parameter is typically added in the implementation to enhance performance \cite{HFL1,AF}. It can intuitively be thought as a units adjustment.
\end{enumerate*}
\end{remark}

\subsection{Iterative subroutine for impulse control}
\label{s:solve_one_player}
Due to Theorem \ref{solution_constrained_QVI_1}, a sensible choice for \textsc{SolveImpulseControl} is the classical policy iteration algorithm \cite[Thm.4.3]{AF} applied to (\ref{restricted_QVI}) (i.e., \cite[Ho-1]{BMZ} applied to (\ref{equivalent_Bellman})), adding an appropriately chosen scaling factor $\lambda$ to improve efficiency (Remark \ref{practical_considerations} \ref{lambda}). One possible drawback of such choice is the following: at each iteration, one needs to solve the system $-A(\tilde\varphi^k)v^{k+1}+b(\tilde\varphi^k)=0$ for some $\tilde\varphi^k\in\tilde\Phi$. While the matrix $\tilde L$ typically has a good sparsity pattern in applications (often tridiagonal), the presence of $\tilde B(\tilde\delta^k)$ prevents $A(\tilde\varphi^k)$ from inheriting the same structure in general, and makes the resolution of the previous system more costly. (See \cite{HFL1} where this issue is addressed for HJB problems with jump diffusions and regime switching, among others.)

Motivated by the previous observation, this section considers an alternative choice for \textsc{SolveImpulseControl}: an instance of a very general class of algorithms known as \textit{fixed-point policy iteration} \cite{HFL1,Cl}. As far as the author knows, this application to impulse control was never done in the past other than heuristically in \cite{ABMZZ}. Instead of solving $-A(\tilde\varphi^k)v^{k+1}+b(\tilde\varphi^k)=0$ at the $k$-th iteration, we will solve 
\begin{equation}
\label{VI}
\underbrace{\left((\tilde{\operatorname{Id}}-\tilde\Psi^k)\tilde L - \tilde\Psi^k\right)}_{-\tilde{\mathbb A}(\tilde\varphi^k)}v^{k+1} + \underbrace{\tilde\Psi^k \tilde B(\tilde\delta^k)}_{\tilde{\mathbb B}(\tilde\varphi^k)}v^k + \underbrace{b(\tilde\varphi^k)}_{\tilde{\mathbb C}(\tilde\varphi^k)}=0,
\end{equation}
(scaled by $\lambda$) where the previous iterate value $v^k$ is given and $\tilde\Psi^k$ is the diagonal matrix with $\psi^k$ as diagonal. In other words, we split the original policy matrix $A(\tilde\varphi)=\tilde{\mathbb A}(\tilde\varphi)-\tilde{\mathbb B}(\tilde\varphi)$ and we apply a one-step fixed-point approximation,
\begin{equation}
\label{FPPI}
\tilde{\mathbb A}\big(\tilde\varphi^k\big)\tilde v^{k+1}=\tilde{\mathbb B}\big(\tilde\varphi^k\big)\tilde v^k +\tilde{\mathbb C}\big(\tilde\varphi^k\big),
\end{equation}
at each iteration of Howard's algorithm. The resulting method can be expressed as follows ($tol$ and $scale$ are as in Algorithm \ref{sym_algo}):

\makeatletter\renewcommand{\ALG@name}{Subroutine}
\begin{algorithm}[H]
    \caption{$(v,I,\delta)=\textsc{SolveImpulseControl}(w,D)$} \label{solve_one_player}
		\textbf{Inputs} $w\in\mathbb R^\grid$ and solvency region $\grid_{\leq 0}\subseteq D\subseteq\grid$
		
		\textbf{Outputs} $v\in\mathbb R^\grid,\ I\subseteq\grid_{<0},\ \delta\in Z$
		\newline
		
		Set scaling factor $\lambda>0$ and $0<\widetilde{tol}\ll tol$. 
		
		\qquad // Restrict constrained problem
		\begin{algorithmic}[1]
			\STATE $\tilde L\defeq L_{DD},\ \tilde f\defeq f_D+L_{DD^c}w_{D^c}$
			\STATE $\tilde Z\defeq \prod_{x\in D}Z(x),\quad \tilde c(\tilde\delta)\defeq c(\tilde\delta)-B(\tilde\delta)_{DD^c}w_{D^c},\ \tilde B(\tilde\delta)\defeq B(\tilde\delta)_{DD}\mbox{ for }\tilde\delta\in\tilde Z$
			\STATE $\tilde M\tilde v\defeq\max_{\tilde\delta\in\tilde Z}\big\{\tilde B(\tilde\delta)\tilde v -\tilde c(\tilde\delta)\big\}\mbox{ for }\tilde v\in\mathbb R^D,\ \tilde{\operatorname{Id}}\defeq \operatorname{Id}_{DD}$
			\newline
					
// Solve restricted problem
			\STATE Choose initial guess: $\tilde v^0\in\mathbb R^D$, $I^0\subseteq\grid_{<0}$
			\FOR{$k=0,1,2,\dots$}
				\STATE $\tilde L^k_{ij}=\begin{cases} 
										\tilde L_{ij} & \mbox{ if }x_i\in D\backslash I^k\\
										-\tilde{\operatorname{Id}}_{ij} & \mbox{ if }x_i\in I^k	
										\end{cases}
										\qquad
								\tilde f^k =\begin{cases} 
										\tilde f & \mbox{ on }\in D\backslash I^k\\
										\tilde M\tilde v^k & \mbox{ on }I^k
										\end{cases}
							 $	
				\STATE $\tilde v^{k+1} \mbox{ solution of }\tilde L^k \tilde v + \tilde f^k=0$
				\STATE $I^{k+1}=\big\{\tilde L\tilde v^{k+1}+\tilde f\leq\lambda\big(\tilde M\tilde v^{k+1}-\tilde v^{k+1}\big)\big\}$
				\IF{$\|(\tilde v^{k+1}-\tilde v^k)/\max\{\tilde v^{k+1},scale\}\|_{\infty}<\widetilde{tol}$} 
				\STATE $v=\begin{cases}
									\tilde v^{k+1} & \mbox{ on }D\\
									w_{D^c} & \mbox{ on }D^c
									\end{cases},
				\ I=I^{k+1},\ \delta=\delta^*(v)$ and break from the iteration 
				\ENDIF	
			\ENDFOR
    \end{algorithmic}
\end{algorithm}
\makeatother

Lines 1-3 of Subroutine \ref{solve_one_player} deal with restricting the constrained problem, while the rest give a routine that can be applied to any QVI of the form (\ref{restricted_QVI}). Starting from some suboptimal $\tilde v^0$ and $I^0$, one computes a new payoff $\tilde v^1$ by solving the coupled equations $\tilde M\tilde v^0-\tilde v^1=0$ on $I^0$ and $\tilde L\tilde v^1+\tilde f=0$ outside $I^0$. A new intervention region $I^1=\big\{\tilde L\tilde v^1+\tilde f\leq\lambda\big(\tilde M\tilde v^1-\tilde v^1\big)\big\}$ is defined and the procedure is iterated.

Algorithmically, the difference with classical policy iteration is that $\tilde v^{k+1}$ is computed in Line 7 with a fixed obstacle $\tilde Mv^k$, changing a quasivariational inequality for a variational one. The resulting method is intuitive and simple to implement, and the linear system (\ref{VI}) (Line 7) inherits the sparsity pattern of $\tilde L$. For example, for an SDD tridiagonal $\tilde L$, the system can be solved (exactly in exact arithmetic or stably in floating point one) in $O(n)$ operations, with $n=|D|$ \cite[Sect.9.5]{H}. The matrix-vector multiply $\tilde B(\tilde\delta^k)\tilde v^k$ can take at most $O(n^2)$ operations, but will reduce to $O(n)$ for standard discretizations of impulse operators.

It is also worth mentioning that Subroutine \ref{solve_one_player} differs from the so-called iterated optimal stopping \cite{COS,OS} in that the latter solves $\max\big\{\tilde L \tilde v^{k+1} + \tilde f, \tilde M \tilde v^k - \tilde v^{k+1}\big\}=0$ exactly at the $k$-th iteration (by running a full subroutine of Howard's algorithm with fixed obstacle), while the former only performs one approximation step. 

To establish the convergence of Subroutine \ref{solve_one_player} in the present framework, we add the following assumption:
\begin{enumerate}[label=(A\arabic*), start=2]
\item \label{A2} $B(\delta)$ has nonnegative diagonal elements for all $\delta\in Z$.
\end{enumerate}

\begin{remark}
\ref{A2} and the requirement of \ref{A1} that $\operatorname{Id}-B(\delta)$ be a WDD $\mbox{L}_0$-matrix are equivalent to $B(\delta)$ being substochastic (see Appendix \ref{appendix:matrices}). This is standard for impulse operators (see Section \ref{s:discretization}) and other applications of fixed-point policy iteration \cite[Sect.4-5]{HFL1}.
\end{remark}

\begin{theorem}
\label{solution_constrained_QVI_2}
Assume \emph{\ref{A0}--\ref{A2}} and set $I^0=\emptyset$. Then, for every $\grid_{\leq 0}\subseteq D\subseteq\grid$ and $w\in\mathbb R^\grid$, the sequence $(\tilde v^k)$ defined by \textsc{SolveImpulseControl}$(w,D)$ is monotone increasing for $k\geq 1$ and converges to the unique solution of \emph{(\ref{restricted_QVI})}.
\end{theorem}

\begin{proof}
We can assume without loss of generality that $\lambda=1$. Subroutine \ref{solve_one_player} takes the form of a fixed-point policy iteration algorithm as per (\ref{FPPI}). Assumptions \ref{A0},\ref{A1} ensure (\ref{restricted_QVI}) has a unique solution (Theorem \ref{solve_one_player}) and that this scheme satisfies \cite[Cond.3.1 (i),(ii)]{HFL1}. That is, $\tilde{\mathbb A}(\tilde\varphi)$ and $\tilde{\mathbb A}(\tilde\varphi)-\tilde{\mathbb B}(\tilde\varphi)$ are nonsingular $M$-matrices (see proof of Theorem \ref{solve_one_player} and Appendix \ref{appendix:matrices}) and all coefficients are bounded since $\tilde\Phi$ is finite. In \cite[Thm.3.4]{HFL1} convergence is proved under one additional assumption of $\|\cdot\|_{\infty}$-contractiveness \cite[Cond.3.1 (iii)]{HFL1}, which is not verified in our case. However, the same computations show that the scheme satisfies
\begin{equation}
\label{aux}
\tilde{\mathbb A}(\tilde\varphi^k)(\tilde v^{k+1}-\tilde v^k)\geq \tilde{\mathbb B}(\tilde\varphi^{k-1})(\tilde v^k-\tilde v^{k-1})\mbox{ for all }k\geq 1.
\end{equation}
Since $I^0=\emptyset$, and due to \ref{A1} and \ref{A2}, $\tilde{\mathbb B}(\tilde\varphi^0)=0$ and $\tilde{\mathbb B}(\tilde\varphi^k)\geq 0$ for all $k$. Thus, $(\tilde v^k)_{k\geq 1}$ is increasing by monotonicity of $\tilde{\mathbb A}(\tilde\varphi^k)$. Furthermore, it must be bounded, since for all $k\geq 1$:
$$
\tilde{\mathbb A}\big(\tilde\varphi^k\big)\tilde v^{k+1}=\tilde{\mathbb B}\big(\tilde\varphi^k\big)\tilde v^k +\tilde{\mathbb C}\big(\tilde\varphi^k\big)\leq\tilde{\mathbb B}\big(\tilde\varphi^{k}\big)\tilde v^{k+1} +\tilde{\mathbb C}\big(\tilde\varphi^k\big),
$$
which gives $\tilde v^{k+1}\leq (\tilde{\mathbb A}(\tilde\varphi^k)-\tilde{\mathbb B}(\tilde\varphi^k))^{-1}\tilde{\mathbb C}\big(\tilde\varphi^k\big)\leq \max_{\tilde\varphi\in\tilde\Phi}(\tilde{\mathbb A}(\tilde\varphi)-\tilde{\mathbb B}(\tilde\varphi))^{-1}\tilde{\mathbb C}\big(\tilde\varphi\big)$. Hence, $(\tilde v^k)$ converges. That the limit solves	(\ref{restricted_QVI}) is proved as in \cite[Lem.3.3]{HFL1}.
\end{proof}

\begin{remark}
\label{r:initial_guess}
In light of Theorem \ref{solution_constrained_QVI_2}, moving forward we will set $I^0=\emptyset$ in Subroutine \ref{solve_one_player}. (Note that the value of $\tilde v^0$ is irrelevant in this case, since $\tilde{\mathbb B}\big(\tilde\varphi^0\big)=0$.) It is natural however to choose $\tilde v^0=w_D$ and $I^0=\big\{\tilde L\tilde v^0+\tilde f\leq\lambda\big(\tilde M\tilde v^0-\tilde v^0\big)\big\}$. The experiments performed with the latter choice displayed faster but non-monotone convergence, but this is not proved here. Additionally, exact convergence was often observed.
\end{remark}

\begin{remark}
\label{r:subroutine}
For the experiments in Section \ref{s:numerics}, \textsc{SolveImpulseControl} was chosen as Subroutine \ref{solve_one_player} instead of classical policy iteration, since the former displayed overall lower runtimes when compared to the latter. However, it should be noted that the games considered have relatively large costs, with many parameter values taken from previous works \cite{ABCCV,BCG,ABMZZ}. For smaller costs often occurring in practice, and especially for large-scale problems, the convergence rate of Subroutine \ref{solve_one_player} can become very slow, as it happens with iterated optimal stopping \cite[Rmk.3.1]{RZ}. This is not so for policy iteration, which converges superlinearly (see \cite[Thm.3.4]{BMZ} and \cite[Sect.5]{RZ}), making it a better suited choice in this case (in terms of runtime as well). Furthermore, it has been demonstrated in \cite{RZ1}[Sect.7] that the number of steps required for \emph{penalized} policy iteration to converge remains bounded as the grid is refined. While penalized methods are not considered in the present work, such mesh-independence properties are desirable and could prove paramount in further studies of Algorithm \ref{sym_algo} (see end of Section \ref{s:convergence} and Section \ref{s:numerics}).
\end{remark}

\subsection{Overall routine as a fixed-point policy-iteration-type method}
\label{s:FPPI}

The system of QVIs (\ref{dQVIs}) cannot be reduced in any apparent way to a Bellman formulation (\ref{Bellman_problem}) (see comments preceding equation). Notwithstanding, we shall see that Algorithm \ref{sym_algo} does take a very similar form to a fixed-point policy iteration algorithm as in (\ref{FPPI}) for some appropriate $\mathbb A,\mathbb B,\mathbb C$. Further, assumptions resembling those of the classical case \cite{HFL1} will be either satisfied or imposed to study its convergence. 
The matrix and graph-theoretic definitions and properties used throughout this section can be found in Appendix \ref{appendix:matrices}.

\begin{notation}
We identify each intervention region $I\subseteq\grid_{<0}$ with its indicator function $\psi=\mathbbm 1_I\in\{0,1\}^\grid$ and each $\psi$ with the diagonal matrix $\Psi=\mbox{diag}(\psi)\in\mathbb R^{\grid\times \grid}$.
The sequences $(v^k)$ and $(\varphi^k)$, with $\varphi^k=(\psi^k,\delta^k)$, are the ones generated by Algorithm \ref{sym_algo}. \textsc{SolveImpulseControl} is defined as either Subroutine \ref{solve_one_player} or Howard's algorithm (Theorem \ref{solution_constrained_QVI_1}), setting the outputs $I,\delta$ as done for the former subroutine. We consider $v^*\in\mathbb R^\grid$ fixed and $\varphi^*=(\psi^*,\delta^*(v^*))$ the induced strategy with $\psi^*\defeq\{Lv^*+f\leq Mv^*-v^*\}\cap\grid_{<0}$.
\end{notation}

\begin{proposition}
\label{FPPI-like_result}
Assume \emph{\ref{A0}--\ref{A2}}. Then, 
\begin{equation}
\label{FPPI-like}
\mathbb A\big(\varphi^k,\varphi^{k+1}\big)v^{k+1}=\mathbb B\big(\varphi^k\big)v^k +\mathbb C\big(\varphi^k,\varphi^{k+1}\big),\mbox{ where:}
\end{equation}
\begin{enumerate}[label=(\roman*)]
\item $\psi^k=\mathbbm 1_{\{Lv^k+f\leq Mv^k-v^k\}\cap\grid_{<0}}\mbox{ and } \delta^k\in\argmax_{\delta\in Z}\big\{B(\delta)v^k-c(\delta)\big\}$.
\item $\mathbb A\big(\varphi,\overline\varphi\big)\defeq \operatorname{Id} -\big(\operatorname{Id} - \overline\Psi-S\Psi S\big)(\operatorname{Id}+L) - \overline\Psi B(\overline\delta)$ is a WCDD $\mbox{L}_0$-matrix, and thus a nonsingular M-matrix.  
\item $\mathbb B\big(\varphi\big)\defeq S\Psi B(\delta)S=diag(S\psi) S B(\delta)S$ is substochastic.
\item $\mathbb C\big(\varphi,\overline\varphi\big)\defeq \big(\operatorname{Id}-\overline\Psi-S\Psi S\big)f - \overline\Psi c(\overline\delta) + S\Psi S g(S\delta)$.
\end{enumerate}
\end{proposition}

\begin{proof}
Using that $(v^{k+1},I^{k+1},\delta^{k+1})=\textsc{SolveImpulseControl}(v^{k+1/2},(-I^k)^c)$ solves the constrained QVI problem (\ref{constrained_QVI}) for $D=(-I^k)^c$ and $w=v^{k+1/2}$ (Theorem \ref{solution_constrained_QVI_1} or \ref{solution_constrained_QVI_2}), the recurrence relation (\ref{FPPI-like}) results from simple algebraic manipulation.

Given $\varphi,\overline\varphi\in\Phi$, \ref{A0} and \ref{A1} ensure $\mathbb A\big(\varphi,\overline\varphi\big)$ is a WCDD $\mbox{L}_0$-matrix, while \ref{A1} and \ref{A2} imply $\mathbb B\big(\varphi\big)$ is substochastic.
\end{proof}

The following corollary is immediate by induction. It gives a representation of the sequence of payoffs in terms of the improving strategies throughout the algorithm.

\begin{corollary}
Assume \emph{\ref{A0}--\ref{A2}}. Then, 
\begin{equation}
\label{sequence_representation}
v^{k+1}=\left(\prod_{j=k}^0\mathbb A^{-1}\mathbb B\big(\varphi^j,\varphi^{j+1}\big)\right)v^0 + \sum_{n=0}^k\left(\prod_{j=k}^{n+1}\mathbb A^{-1}\mathbb B\big(\varphi^j,\varphi^{j+1}\big)\right)\mathbb A^{-1}\mathbb C\big(\varphi^n,\varphi^{n+1}\big).\footnote{For any index $i\leq k$, $\prod_{j=k}^iA^j\defeq A^kA^{k-1}\dots A^i$.}
\end{equation}
\end{corollary}

We now establish some properties of the strategy-dependent matrix coefficients that will be useful in the sequel. Given a WDD (resp. substochastic) matrix $A\in \mathbb R^{\grid\times\grid}$, we define its set of ``non-trouble states" (or rows) as 
$$
J[A]\defeq \{x\in\grid:\mbox{ row }x\mbox{ of }A\mbox{ is SDD}\}\mbox{ (resp. }\hat J[A]\defeq \{x\in\grid:\mbox{ row }x\mbox{ of }A\mbox{ sums less than one}\}),
$$ 
and its \textit{index of connectivity} $\mbox{con}A$ (resp. \textit{index of contraction} $\widehat{\mbox{con}}A$) by computing for each state the least length that needs to be walked on graph$A$ to reach a non-trouble one, and then taking the maximum over all states (more details in Appendix \ref{appendix:matrices}). This recently introduced concept gives an equivalent charaterization of the WCDD property for a WDD matrix as con$A<+\infty$, and can be efficiently checked for sparse matrices in $O(|\grid|)$ operations \cite{A}. On the other hand, if $A$ is substochastic then $\widehat{\mbox{con}}A<\infty$ if and only if its spectral radius verifies $\rho(A)<1$ (Theorem \ref{contraction}). The proof of the following lemma can be found in Appendix \ref{proof:lemma}.

\begin{lemma}
\label{coefficients_properties}
Assume \emph{\ref{A0}--\ref{A2}}. Then for all $\varphi,\overline\varphi\in\Phi$, $\mathbb A^{-1}\mathbb B\big(\varphi,\overline\varphi\big)$ is substochastic, $(\mathbb A-\mathbb B)\big(\varphi,\overline\varphi\big)$ is a WDD $L_0$-matrix and $\widehat{\emph{con}}\big[\mathbb A^{-1}\mathbb B\big(\varphi,\overline\varphi\big)\big]\leq\emph{con}\big[(\mathbb A-\mathbb B)\big(\varphi,\overline\varphi\big)\big]$.
\end{lemma}

As previously mentioned, system (\ref{dQVIs}) may have no solution. The matrix coefficients introduced in this section allow us to algebraically characterize the existence of such solutions through strategy-dependent linear systems of equations. 

\begin{proposition}
\label{equivalent_dQVIs1}
Assume \emph{\ref{A0}--\ref{A2}}. 
Then the following statements are equivalent:
\begin{enumerate}[label=(\roman*)]
\item $v^*$ solves the system of QVIs (\ref{dQVIs}).
\item $\mathbb A\big(\varphi^*,\varphi^*\big)v^*=\mathbb B\big(\varphi^*\big)v^* +\mathbb C\big(\varphi^*,\varphi^*\big)$.
\end{enumerate}
\end{proposition}

As mentioned in Remark \ref{r:interpretation}, Assumption \ref{A0} constrains the type of strategies the player can use, but without taking into account the opponent's response. This is enough for the single-player constrained problems to have a solution and, therefore, for Algorithm \ref{sym_algo} to be well defined. But we cannot expect this restriction to be sufficient in the study of the two-player game and the convergence of the overall routine. 

In order to improve the result of Proposition \ref{equivalent_dQVIs1} 
let us consider the following stronger version of \ref{A0} reflecting the interaction between the player and the opponent.

\begin{enumerate}[label=(A\arabic*'),start=0]
\item \label{A0'} 
For each pair of strategies $\varphi,\overline\varphi\in\Phi$, and for each $x\in \overline I\cup (-I)$, there exists a walk in graph$\big(\overline\Psi B(\overline\delta)+S\Psi B(\delta)S\big)$ from row $x$ to some row $y\in \overline C\cap  C$, where $\overline C={\overline I}^c,\ C=I^c$.
\end{enumerate} 

\begin{remark}(\textit{Interpretation})
\label{r:interpretation2}
If $\overline\varphi,-\varphi$ are the strategies used by the player and the opponent respectively,\footnote{The slight abuse of notation $-\varphi$ stands for the strategy symmetric to $\varphi$, i.e., $-\varphi=(-I,-\delta(-x))$.} then \ref{A0'} asserts that states in their intervention regions will eventually be shifted to the common continuation region. This precludes infinite simultaneous interventions and emulates the admissibility condition of the continuous-state case. Fixing $I=\emptyset$ we recover \ref{A0}. Additionally, \ref{A0'} together with \ref{A1} imply that $(\mathbb A-\mathbb B)(\varphi,\overline\varphi)$ is a WCDD $\mbox{L}_0$-matrix,\footnote{\ref{A1} guarantees the WDD $\mbox{L}_0$-property (see proof of Lemma \ref{coefficients_properties}) with SDD $i$-th row for any $x_i\in\overline C\cap C$, as it coincides with the $i$-th row of $-L$. Since $(\mathbb A-\mathbb B)(\varphi,\overline\varphi)$ and $\operatorname{Id} - \big(\overline\Psi B(\overline\delta)+S\Psi B(\delta)S\big)$ agree for any row indexed in $\overline I\cup (-I)$ and graph$\big(\operatorname{Id} - \big(\overline\Psi B(\overline\delta)+S\Psi B(\delta)S\big)\big)$ is equal to graph$\big(\overline\Psi B(\overline\delta)+S\Psi B(\delta)S\big)$ except possibly for self loops, adding \ref{A0'} means that any WDD row of $(\mathbb A-\mathbb B)(\varphi,\overline\varphi)$ is connected to some SDD row.} hence an $M$-matrix. This is another one of the assumptions of the classical fixed-point policy iteration \cite{HFL1}.
\end{remark}

Under this new assumption, the $\varphi^*=\varphi^*(v^*)$-dependent systems of Proposition \ref{equivalent_dQVIs1} will admit a unique solution. Then solving the original problem (\ref{dQVIs}) amounts to finding $v^*\in\mathbb R^\grid$ that solves its induced linear system of equations.
\begin{proposition}
\label{equivalent_dQVIs2}
Assume \emph{\ref{A0'},\ref{A1},\ref{A2}}. In the context of Proposition \ref{equivalent_dQVIs1}, the following statements are also equivalent:
\begin{enumerate}[label=(\roman*),start=4]
\item $v^*=(\mathbb A-\mathbb B)^{-1}\mathbb C\big(\varphi^*,\varphi^*\big)$.
\item $v^*= (\operatorname{Id}-\mathbb A^{-1}\mathbb B)^{-1}\mathbb A^{-1}\mathbb C\big(\varphi^*,\varphi^*\big)=\sum_{n\geq 0}\big(\mathbb A^{-1}\mathbb B\big)^n\mathbb A^{-1}\mathbb C\big(\varphi^*,\varphi^*\big)$. (cf. equation (\ref{sequence_representation}).)
\end{enumerate}
\end{proposition}

\begin{proof}
Both expressions result from rewriting and solving the systems of Proposition \ref{equivalent_dQVIs1}. Assumptions \ref{A0'},\ref{A1} guarantee that $(\mathbb A-\mathbb B)\big(\varphi^*,\varphi^*\big)$ is WCDD (Remark \ref{r:interpretation2}) and, in particular, nonsingular. Then {\textit{(v)}} is due to Lemma \ref{coefficients_properties}, Theorem \ref{contraction} and the matrix power series expansion $(\operatorname{Id}-X)^{-1}=\sum_{n\geq 0}X^n$, when $\rho(X)<1$.  
\end{proof}

\subsection{Convergence analysis}
\label{s:convergence}

We now study the convergence properties of Algorithm \ref{sym_algo}. Henceforth,  the UIP refers to the obvious discrete analogous of Definition \ref{UIP}, where we replace the domain $\mathbb R$, the impulse constraints $\mathcal Z$ and the operator $\mathcal M$ by their discretizations $\grid,\ Z$ and $M$ respectively.

The obvious first question to address is whether when Algorithm \ref{sym_algo} converges, it does so to a solution of the system of QVIs (\ref{dQVIs}). Unlike in the classical Bellman problem (\ref{Bellman_problem}), problem (\ref{dQVIs}) is intrinsically dependent on the particular strategy chosen by the player (see Propositions \ref{equivalent_dQVIs1} and \ref{equivalent_dQVIs2}). Accordingly, we start with a lemma addressing what can be said about the convergence of the strategies $(\varphi^k)$ when the payoffs $(v^k)$ converge.

\begin{notation}
$\partial I^*\defeq \{Lv^*+f=Mv^*-v^*\}\cap\grid_{<0}$ denotes the ``border" of the intervention region $\{Lv^*+f\leq Mv^*-v^*\}\cap\grid_{<0}$ defined by $v^*$.
\end{notation}
  
\begin{lemma}
\label{convergence_strategies}
Assume \emph{\ref{A0}-\ref{A2}} and suppose $v^k\to v^*$. Then:
\begin{enumerate}[label=(\roman*)]
\item $\psi^k\to\psi^*$ in $(\partial I^*)^c$ and $Mv^k\to Mv^*$.
\item If $\overline\psi,\overline\delta$ are any two limit points of $(\psi^k),(\delta^k)$ resp.,\footnote{By ``limit point" we mean the limit of a convergent subsequence.} then 
$$
\overline\delta\in\argmax_{\delta\in Z}\big\{B(\delta)v^*-c(\delta)\big\},\quad\overline\psi=0\mbox{ on }\grid_{>=0}\quad\mbox{and}\quad\overline\psi\in\argmax_{i\in\{0,1\}}\big\{O_iv^*\big\}\mbox{ on }\grid_{<0},
$$
with $O_0v=Lv+f$ and $O_1v=Mv-v$.
\item If $v^*$ has the UIP, then $\delta^k\to\delta^*(v^*)$ and $Hv^k\to Hv^*$.
\end{enumerate}
\end{lemma}

\begin{proof}
That $Mv^k\to Mv^*$ is clear by continuity of the operators $B(\delta)$ and finiteness of $Z$. 

Let $x\in(\partial I^*)^c$ and suppose $Lv^*(x)+f(x)<Mv^*(x)-v^*(x)$ (the other case being analogous). By continuity of $L$ and $M$ there must exist some $k_0$ such that $Lv^k(x)+f(x)<Mv^k(x)-v^k(x)$ for all $k\geq k_0$, which implies $\psi^k(x)=1=\psi^*(x)$ for $k\geq k_0$.

The statement about $\overline\psi,\overline\delta$ is proved as before by considering appropriate subsequences. Consequently, if $v^*$ has the UIP, then necessarily $\delta^k\to\delta^*(v^*)$ and $Hv^k\to Hv^*$.
\end{proof}

As a corollary we can establish that, should the sequence $(v^k)$ converge, its limit must solve problem (\ref{dQVIs}). If convergence is not exact however (i.e., in finite iterations), then we will ask that $v^*$ verifies some of the properties of the Verification Theorem in Corollary \ref{coro_sym_QVIs}. Namely, the UIP and a discrete analogous of the continuity in the border of the opponent's intervention region. 
We emphasise that our main motivation in solving system (\ref{dQVIs}) relies in Corollary \ref{coro_sym_QVIs} and its framework. Additionally, in most practical situations and for fine-enough grids, one can intuitively expect the discretization of an equilibrium payoff as in Corollary \ref{coro_sym_QVIs} to inherit the UIP. Lastly, we note that the exact equality $Lv^*+f=Mv^*-v^*$ will typically not be verified for any point in the grid in practice, giving $\partial I^*=\emptyset$.

\begin{corollary}
\label{convergence_to_solution}
Assume \emph{\ref{A0}--\ref{A2}} and suppose $v^k\to v^*$. Then:
\begin{enumerate}[label=(\roman*)]
\item If the convergence is exact, then $v^*$ solves the system of QVIs (\ref{dQVIs}).
\item If $v^*$ has the UIP and $Lv^*+f=Hv^*-v^*$ on $-\partial I^*$, then $v^*$ solves (\ref{dQVIs}).
\end{enumerate}
\end{corollary}

\begin{proof}
\textit{(i)} is immediate from the definition of Algorithm \ref{sym_algo}.

In the general case, since $\{0,1\}^\grid$ is finite, there is a subsequence of $\big(\psi^k,\psi^{k+1}\big)$ that converges to some pair $(\psi,\overline\psi)$. Passing to such subsequence, by Lemma \ref{convergence_strategies}, the UIP of $v^*$ and equation (\ref{FPPI-like}), we get that $v^*$ solves the system $\mathbb A\big(\varphi,\overline\varphi\big)v^*=\mathbb B\big(\varphi\big)v^* +\mathbb C\big(\varphi,\overline\varphi\big)$ for $\varphi=\big(\psi,\delta^*(v^*)\big),\overline\varphi=\big(\overline\psi,\delta^*(v^*)\big)$ and $\psi,\overline\psi$ coincide with $\psi^*$ except possibly on $\partial I^*$. Thus, it only remains to show that $v^*$ also solves the equations of the system (\ref{dQVIs}) for any $x\in\partial I^*\cup(-\partial I^*)$. 

For $x\in\partial I^*$, the previous is true by definition. Suppose now $x\in-\partial I^*\subseteq {\overline I}^c$. We have $\psi^*(-x)=1$. If $\psi(-x)=1$, there is nothing to prove. If $\psi(-x)=0$, then $x\in{\overline I}^c\cap(-I)^c$ and $0=Lv^*(x)+f(x)=Hv^*(x)-v^*(x)$, 
where the last equality holds true by assumption.
\end{proof}

Lemma \ref{convergence_strategies} shows to what extent the convergence of the payoffs imply the convergence of the strategies. The following theorem, of theoretical interest, establishes a reciprocal under the stronger assumption \ref{A0'}. In general, since the set of strategies $\Phi$ is finite, the sequence of strategy-dependent coefficients of the fixed-point equations (\ref{FPPI-like}) will always be bounded and with finitely many limit points. However, if the approximating strategies are such that the former coefficients convergence, then Algorithm \ref{sym_algo} is guaranteed to converge. Further, instead of looking at the convergence of $\big(\mathbb A,\mathbb B,\mathbb C\big)\big(\varphi^k,\varphi^{k+1}\big)$, we can instead consider the weaker condition of $\big(\mathbb A^{-1}\mathbb B,\mathbb A^{-1}\mathbb C\big)\big(\varphi^k,\varphi^{k+1}\big)$ converging.

\begin{theorem}
\label{convergence_payoffs}
Assume \emph{\ref{A0'},\ref{A1},\ref{A2}}. If $\big(\mathbb A^{-1}\mathbb B\big(\varphi^k,\varphi^{k+1}\big)\big)$ and $\big(\mathbb A^{-1}\mathbb C\big(\varphi^k,\varphi^{k+1}\big) \big)$ converge, then $(v^k)$ converges.
\end{theorem}
\begin{proof}
Set $b=\lim_k \mathbb A^{-1}\mathbb C\big(\varphi^k,\varphi^{k+1}\big)$. Since $\Phi$ is finite, there must exist $k_0\in\mathbb N$ and $\varphi,\overline\varphi\in\Phi$ such that $\mathbb A^{-1}\mathbb B\big(\varphi^k,\varphi^{k+1}\big)=\mathbb A^{-1}\mathbb B\big(\varphi,\overline\varphi\big)$ and $\mathbb A^{-1}\mathbb C\big(\varphi^k,\varphi^{k+1}\big)=b$ for all $k\geq k_0$. Moreover, under our assumptions, $(\mathbb A-\mathbb B)\big(\varphi,\overline\varphi\big)$ is a WCDD $\mbox{L}_0$-matrix. Then Lemma \ref{coefficients_properties} and Theorem \ref{contraction} imply that $\mathbb A^{-1}\mathbb B\big(\varphi,\overline\varphi\big)$ is contractive for some matrix norm. Lastly, note that the sequence of payoffs $(v^k)_{k\geq k_0}$ now satisfies the classical (constant-coefficients) contractive fixed-point recurrence $v^{k+1}= \mathbb A^{-1}\mathbb B\big(\varphi,\overline\varphi\big)v^k + b$, which converges to the unique fixed-point of the equation.
\end{proof}

The classical fixed-point policy-iteration framework \cite{HFL1,Cl} assumes uniform contractiveness in $\|\cdot\|_\infty$ of the sequence of operators. This is a natural norm to consider in a context where matrices have properties defined row by row, such as diagonal dominance.\footnote{Recall that this norm can be computed as the maximum absolute value row sum.} However, the authors mention convergence in experiments where only $\|\cdot\|_\infty$-nonexpansiveness held true. The latter is the typical case in our context, for the matrices $\mathbb A^{-1}\mathbb B\big(\varphi^k,\varphi^{k+1}\big)$, which is why Theorem \ref{convergence_payoffs} relies on the fact that a spectral radius strictly smaller than one guarantees contractiveness in some matrix norm. 

It is natural to ask whether there is some contractiveness condition that may account for the observations in \cite{HFL1,Cl} and that can be generalized to our context to further the study of Algorithm \ref{sym_algo}. Imposing a uniform bound on the spectral radii would not only be hard to check, but also difficult to manipulate, as the spectral radius is not sub-multiplicative.\footnote{$\rho(AB)\leq \rho(A)\rho(B)$ does not hold in general when the matrices $A$ and $B$ do not commute.} Instead, we can consider the sequential indices of contraction and connectivity, which naturally generalize those of the previous section by means of walks in the graph of a sequence of matrices (see Appendix \ref{appendix:matrices} for more details). As before, they can be identified with one another (see Lemma \ref{substochastic_WDD_link}) and, given substochastic matrices, the sequential index of contraction tells us how many we need to multiply before the result becomes $\|\cdot\|_\infty$-contractive (Theorem \ref{sequence_contraction}). Thus, let us consider a uniform bound on the following sequential indices of connectivity: 

\begin{enumerate}[label=(A\arabic*''),start=0]
\item \label{A0''} 
There exists $m\in\mathbb N_0$ such that for any sequence of strategies $(\overline\varphi^k)\subseteq\Phi$, 
$$
\mbox{con}\left[\Big((\mathbb A-\mathbb B)\big(\overline\varphi^k,\overline\varphi^{k+1}\big)\Big)\right]\leq m.
$$
\end{enumerate} 

\begin{remark}
Given $\varphi,\overline\varphi\in\Phi$, by considering the sequence $\varphi,\overline\varphi,\varphi,\overline\varphi,\dots$, we see that \ref{A0''} implies \ref{A0'}. In fact, \ref{A0''} can be interpreted as precluding infinite simultaneous impulses even when the players can adapt their strategies (cf. Remark \ref{r:interpretation2}) and imposing that the number of shifts needed for any state to reach the common continuation region is bounded.
\end{remark}

Under this stronger assumption, we have:
\begin{proposition}
\label{bounded_iterates}
Assume \emph{\ref{A0''},\ref{A1},\ref{A2}}. Then $(v^k)$ is bounded.
\end{proposition}

\begin{proof}
In a similar way to Lemma \ref{coefficients_properties}, one can check that under \ref{A0''},\ref{A1},\ref{A2} we have the following uniform bound for the sequential indices of contraction:
$$
\widehat{\mbox{con}}\left[\Big(\mathbb A^{-1}\mathbb B\big(\varphi^k,\varphi^{k+1}\big)\Big)\right]\leq m,
$$
for any sequence of strategies $(\varphi^k)\subseteq\Phi$. In other words, not only is any product of the previous substochastic matrices also substochastic (i.e., $\|\cdot\|_\infty$-nonexpansive), but it is also $\|.\|_{\infty}$-contractive when there are at least $m+1$ factors. Furthermore, since $\Phi^{m+1}$ is finite, there must be a uniform contraction constant $C_1<1$. Let $C_2>0$ be a uniform bound for $\mathbb A^{-1}\mathbb C$. By the representation (\ref{sequence_representation})
\begin{equation*}
\|v^{k+1}\|_\infty\leq \|v^0\|_\infty + C_2\sum_{n=0}^k C_1^{\big[\frac{k-n}{m+1}\big]}
\leq\|v^0\|_\infty + (m+1)C_2\sum_{n=0}^\infty C_1^n<+\infty.\footnote{For any $x\in\mathbb R$, $[x]$ denotes its integer part.}
\end{equation*}
\end{proof}

Given $n_0\in\mathbb N$ and $k>(m+1)n_0$, the same argument of the previous proof shows that one can decompose $(v^k)$ as 
$
v^{k+1}=u^k + F(\varphi^{k-(m+1)n_0},\dots,\varphi^k) + w^k,
$
 for a fixed function $F$, $\|u^k\|_\infty\leq C_1^{[k/(m+1)]}\|v^0\|\to 0$ and $\|w^k\|_\infty\leq (m+1)C_2\sum_{n=n_0}^\infty C_1^n$. The latter is small if $n_0$ is large. Hence, one could heuristically expect that the trailing strategies are often the ones dominating the convergence of the algorithm. In fact, in all the experiments carried out with a discretization satisfying \ref{A0''},\ref{A1},\ref{A2}, a dichotomous behaviour was observed: the algorithm either converged or at some point reached a cycle between a few payoffs. In the latter case, and restricting attention to instances in which one heuristically expects a solution to exist (more details in Section \ref{s:numerics}), it was possible to reduce the residual to the QVIs and the distance between the iterates by refining the grid. 

The previous motivates the study of Algorithm \ref{sym_algo} when the grid is sequentially refined, instead of fixed. Such an analysis however, would likely entail the need of a viscosity solutions framework as in \cite{ABL,BS}, which does not currently exist in the literature of nonzero-sum stochastic impulse games. 
Consequently, this analysis and the stronger convergence results that may come out of it are inevitably outside the scope of this paper. 

\subsection{Discretization schemes}
\label{s:discretization}
Let us conclude this section by showing how one can discretize the symmetric system of QVIs (\ref{sym_QVIs}) to obtain (\ref{dQVIs}) in a way that satisfies the assumptions present throughout the paper. Recall that we work on a given symmetric grid $\grid:\ x_{-N}=-x_N<\dots<x_{-1}=-x_1<x_0=0<x_1<\dots<x_N$.

Firstly, we want a discretization $L$ of the operator $\mathcal A-\rho \operatorname{Id}$ such that $-L$ is an SDD $L_0$-matrix as per \ref{A1}. A standard way to do this is to approximate the first (resp. second) order derivatives with forward and backward (resp. central) differences in such a way that we approximate the ordinary differential equation (ODE) $\frac{1}{2}\sigma^2 V''+\mu V' - \rho V + f=0$ 
with an \textit{upwind} (or \textit{positive-coefficients}) scheme. More precisely, for each $x=x_i\in\grid$ we approximate the first derivative with a forward (resp. backward) difference if its coefficient in the previous equation is nonegative (resp. negative) in $x_i$, 
\begin{equation*}
V'(x_i)\approx\frac{V(x_{i+1})-V(x_i)}{x_{i+1}-x_i}\quad\mbox{ if }\mu(x_i)\geq 0\quad\mbox{ and }\quad V'(x_i)\approx\frac{V(x_i)-V(x_{i+1})}{x_i-x_{i+1}}\quad\mbox{ if }\mu(x_i)<0					
\end{equation*}
and the second derivative by
$$
V''(x_i)\approx\frac{V(x_{i+1})-V(x_i)}{(x_{i+1}-x_i)(x_{i+1}-x_{i-1})}-\frac{V(x_i)-V(x_{i-1})}{(x_i-x_{i-1})(x_{i+1}-x_{i-1})}.						
$$
In the case of an equispaced grid with step size $h$, this reduces to 
\begin{equation*}
V'(x)\approx\frac{V\big(x+\sgn(\mu(x))h\big)-V(x)}{\sgn(\mu(x))h}					
\quad\mbox{ and }\quad		
V''(x)\approx\frac{V(x+h)-2V(x)+V(x-h)}{h^2}.\footnote{$\sgn$ denotes the sign function, $\sgn(x)=1$ if $x\geq 0$ and $-1$ otherwise.}						
\end{equation*}
For the previous stencils to be defined in the extreme points of the grid, we consider two additional points $x_{-N-1},x_{N+1}$ and replace $V(x_{-N-1}),V(x_{N+1})$ in the previous formulas by some values resulting from artificial boundary conditions. A common choice is to impose Neumann conditions to solve for $V(x_{-N-1}),V(x_{N+1})$ using the first order differences from before. For example, in the equispaced grid case, 
given $\mbox{LBC},\mbox{RBC}\in\mathbb R$ we solve for $V(x_{-N}-h)$ (resp. $V(x_N+h)$) from the Neumann condition 
$$\mbox{LBC}=V'\left(x_{-N}-h\mathbbm 1_{\big\{\mu(x_{-N})\geq 0\big\}}\right)\quad \mbox{(resp. } \mbox{RBC}=V'\left(x_N+h\mathbbm 1_{\big\{\mu(x_N)< 0\big\}}\right)),$$
yielding $V(x_{-N}-h)\approx V(x_{-N})-\mbox{LBC}h$ (resp. $V(x_N+h)\approx V(x_N)+ \mbox{RBC} h$). 
The choice of $\mbox{LBC},\mbox{RBC}$ is problem-specific and intrinsically linked to that of $x_N$, although it does not affect the properties of the discrete operators. See more details in Section \ref{s:numerics}.

The described procedure leads to a discretization of the ODE as $Lv+f=0$, with $L$ satisfying the properties we wanted (with \emph{strict} diagonal dominance due to $\rho>0$). Note that the values of $f$ at $x_{-N},x_N$ need to be modified to account for the boundary conditions. 

\begin{remark}
One could increase the overall order of approximation by using central differences as much as possible for the first order derivatives, provided the scheme remains upwind (see \cite{FL, WF} for more details). This is not done here in order to simplify the presentation. 
\end{remark}

We now approximate the impulse constraint sets $\mathcal Z(x)$ ($x\in\mathbb R$) by finite sets $\emptyset\neq Z(x)\subseteq[0,+\infty)$ ($x\in\grid$), such that $Z(x)=\{0\}$ if $x\geq 0$, and define the impulse operators 
$$B(\delta)v(x) = v[\![x+\delta(x)]\!],\mbox{ for }v\in\mathbb R^\grid,\ \delta\in Z,\ x\in\grid,$$
where $v[\![y]\!]$ denotes linear interpolation of $v$ on $y$ using the closest nodes on the grid, and $v[\![y]\!]=v(x_{\pm N})$ if $\pm y>\pm x_{\pm N}$ (i.e., ``no extrapolation"). This univocally defines the discrete loss and gain operators $M$ and $H$ as per (\ref{discrete_intervention_operators}), as well as the optimal impulse $\delta^*$ according to (\ref{delta_star}). The set of discrete strategies $\Phi$ is defined as in (\ref{discrete_strategies}). 

This general discretization scheme satisfies assumptions \ref{A0}--\ref{A2} and one can impose some regularity conditions on the sets $\mathcal Z(x)$ and $Z(x)$ such that the solutions of the discrete QVI problems (\ref{constrained_QVI}) converge locally uniformly to the unique viscosity solution of the analytical impulse control problem, as the grid is refined.\footnote{Additional technical conditions include costs bounded away from zero and a comparison principle for the analytical QVI.} See \cite{A0,ABL} for more details.

\begin{example}
\label{ex:arbitrary_impulses}
In the case where $\mathcal Z(x)=[0,+\infty)$ for $x<0$, a natural and most simple choice for $Z(x)$ is $Z(x_i)=\{0,x_{i+1}-x_i,\dots,x_N-x_i\}$ for $i<0$. In this case, $B(\delta)v(x) = v(x+\delta(x))$ and $Hv(x)=v(x-\delta^*(-x))+g(x,-\delta^*(-x))$. This choice, however, does not satisfy \ref{A0'}. 
\end{example}

In order to preclude infinite simultaneous interventions it is enough to constrain the size of the impulses so that the symmetric point of the grid cannot be reached. That is, $Z(x)\subseteq[0,-2x)$ for any $x\in\grid_{<0}$. In this case, the scheme satisfies the stronger conditions \ref{A0''},\ref{A1},\ref{A2} (and in particular, \ref{A0'}). Intuitively, each positive impulse will lead to a state which is at least one node closer to $x_0=0$, where no player intervenes. Practically, it makes sense to make this choice when one suspects (or wants to check whether) there is a symmetric NE with no ``far-reaching impulses", in the previous sense.

\begin{proof}
To check that \ref{A0''} is indeed satisfied, consider some arbitrary $\varphi=(I,\delta),\ \overline\varphi=(\overline I,\overline\delta)$ and $x_i\in\grid$. Note first that if $x_i$ belongs to the common continuation region ${\overline I}^c\cap (-I)^c$ (as it is the case for $x_0=0$), then the $i$-th row of $(\mathbb A-\mathbb B)\big(\varphi,\overline\varphi\big)$ is equal to the $i$-th row of $-L$, which is SDD by \ref{A1}. We claim that, in general, either the $i$-th row of $(\mathbb A-\mathbb B)\big(\varphi,\overline\varphi\big)$ is SDD or there is an edge in its graph from $x_i$ to some $x_j$ with $|j|<|i|$. This immediately implies that (in the notation of \ref{A0''}) one can take $m=N$, as there will always be a walk in graph$\Big((\mathbb A-\mathbb B)\big(\overline\varphi^k,\overline\varphi^{k+1}\big)\Big)$ from $x_i$ to some SDD row, with length at most $N$. (The longest that might need to be walked is all the way until reaching the 0-th row, with the distance decreasing in one unit per step.) 
 
To prove the claim, suppose first that $x_i\in\overline I$ (i.e., the intervention region of the player). Hence, $i<0$. If $\overline\delta=0$, then the $i$-th row of $(\mathbb A-\mathbb B)\big(\varphi,\overline\varphi\big)$ is again equal to the $i$-th row of $\operatorname{Id}$, thus SDD. If $0<\overline\delta<-2x_i$, then there exist an index $k$ and $\alpha,\beta\geq 0$ such that $i< k\leq |i|$, $\alpha+\beta=1$ and $x_i+\overline\delta=\alpha x_{k-1} +\beta x_k$. Accordingly, 
$$\left[(\mathbb A-\mathbb B)\big(\varphi,\overline\varphi\big)\right]_{ij} = \mathbbm 1_{\{i=j\}} - \alpha\mathbbm 1_{\{j=k-1\}} - \beta\mathbbm 1_{\{j=k\}},\quad\mbox{for all }j.$$
If $x_i<x_i+\overline\delta\leq x_{|i|-1}$, then we can assume without loss of generality that $\beta>0$ and take $j=k$. (If $\beta=0$, simply relabel $k-1$ as $k$.) If $x_{|i|-1}<x_i+\overline\delta< -x_i = x_{|i|}$, then $\alpha>0$ and we can take $j=k-1$.

The case of $x_i\in -I$ is symmetric and can be proved in the same manner.
\end{proof}

\begin{example}
\label{ex:impulses_constrained_size}
If $\mathcal Z(x)=[0,+\infty)$ for $x<0$, the analogous of Example \ref{ex:arbitrary_impulses} is now $Z(x_i)=\{0,x_{i+1}-x_i,\dots,x_{-i-1}-x_i\}$ for $i<0$. 
\end{example}

\begin{remark}
\label{r:maxargmax2}
Consider Example \ref{ex:impulses_constrained_size} in the context of Remark \ref{r:maxargmax1}. As in Proposition \ref{bounded_iterates} and due to Theorem \ref{sequence_contraction}, the less impulses needed between the two players to reach the common continuation region, the faster that the composition of the fixed-point operators of Algorithm \ref{sym_algo} becomes contractive. Hence, one could intuitively expect that when close enough to the solution, the choice of the maximum arg-maximum in (\ref{delta_star}) improves the performance of Algorithm \ref{sym_algo}. This is another motivation for such choice.  
\end{remark}

\section{Numerical results}
\label{s:numerics}
This section presents numerical results obtained on a series of experiments. See Introduction and Section \ref{s:symmetric_games} for the motivation and applications behind some of them. 
We do not assume additional constraints on the impulses in the analytical problem. 
All the results presented were obtained on equispaced grids with step size $h>0$ (to be specified) and with a discretization scheme as in Section \ref{s:discretization} and Example \ref{ex:impulses_constrained_size}. The extreme points of the grid are displayed on each graph. 

For the games with linear costs and gains of the form $c(x,\delta)=c^0+c^1\delta$ and $g(x,\delta)=g^0+g^1\delta$, with $c^0,c^1,g^0,g^1$ constant, the artificial boundary conditions were taken as LBC $=c^1$ and RBC $=g^1$ for a sufficiently extensive grid. They result from the observation that on a hypothetical symmetric NE of the form $\varphi^*=\big((-\infty, \overline x], \delta^*(x)=y^*-x\big)$, with $\overline x<0,\ \overline x< y^*\in\mathbb R$, the equilibrium payoff verifies $V(x)=V(y^*)-c^0-c^1(y^*-x)$ for $x<\overline x$ and $V(x)=V(-y^*)+g^0+g^1(x+y^*)$ for $x>-\overline x$. For other examples, LBC, RBC and the grid extension were chosen by heuristic guesses and/or trial and error. However, in all the examples presented the error propagation from poorly chosen LBC,RBC was minimal. 

The initial guess was set as $v_0=0$ and its induced strategy in all cases. \textsc{SolveImpulseControl} was chosen as Subroutine \ref{solve_one_player} (see Remark \ref{r:subroutine}) with $\widetilde {tol}=10^{-15}$ and $\lambda=1$.\footnote{For very fine grids, one should increase the value of $\widetilde {tol}$ to avoid stagnation (see, e.g, Remark \ref{practical_considerations}).} Its convergence was exact however, in all the examples. Instead of fixing a terminal tolerance $tol$ beforehand, we display the highest accuracy attained and number of iterations needed.

Section \ref{s:games_on_fixed_grid} considers a fixed grid and games where the results point to the existence of a symmetric NE as per Corollary \ref{coro_sym_QVIs}. Not having an analytical solution to compare with, results are assessed by means of the percentage difference between the iterates 
$$\mbox{Diff}\defeq\left\|(v^{k+1}-v^k)/\max\{|v^{k+1}|,scale\}\right\|_{\infty},$$ 
with $scale=1$ as in \cite{AF}, and the maximum pointwise residual to the system of QVIs (\ref{dQVIs}), defined for $v\in\mathbb R^\grid$ by setting $I=\{Lv+f\leq Mv-v\}\cap\grid_{<0}$, $C=I^c$ and
$$
\mbox{maxResQVIs}(v)\defeq\left\|  \max\{Lv+f,Mv-v\}\mathbbm 1_{-C} + (Hv-v)\mathbbm 1_{-I} \right\|_{\infty}.
$$
Section \ref{cv_analytical_sol} considers the only symmetric games in the literature with (semi-) analytical solution: the central bank linear game \cite{ABCCV} and the cash management game \cite{BCG}, and computes the errors made by discrete approximations, showing in particular the effect of refining the grid. Not considered here is the strategic pollution control game \cite{FK}, due to its inherent non-symmetric nature. Section \ref{s:games_without_NE} comments on results obtained for games without NEs. Finally, Section \ref{s:beyond_verif_theo} shows results that go beyond the scope of the currently available theory for impulse games. 

\subsection{Convergence to discrete solution on a fixed grid}
\label{s:games_on_fixed_grid}

Throughout this section the grid step size is fixed as $h=0.01$, unless otherwise stated (although results where corroborated by further refinements). Each figure specifies the structure, $\mathcal G=(\mu,\sigma,\rho,f,c,g)$, of the symmetric game solved and shows the numerical solutions at the terminal iteration for the equilibrium payoff, $v^k$, and NE. Graphs plot payoff versus state of the process. The intervention region is displayed in red over the graph of the payoff for presentation purposes.

We focus on games with higher costs than gains, as the opposite typically leads to players attempting to apply infinite simultaneous impulses \cite{ABCCV} (i.e., inducing a gain from the opponent's intervention is ``cheap") leading to degenerate games. The following games resulted in exact convergence in finite iterations, which guarantees a solution of (\ref{dQVIs}) was reached (Corollary \ref{convergence_to_solution}), although very small acceptable errors were reached sooner.

\begin{figure}[H]
\hspace*{.2cm}
	\includegraphics[scale=.36]{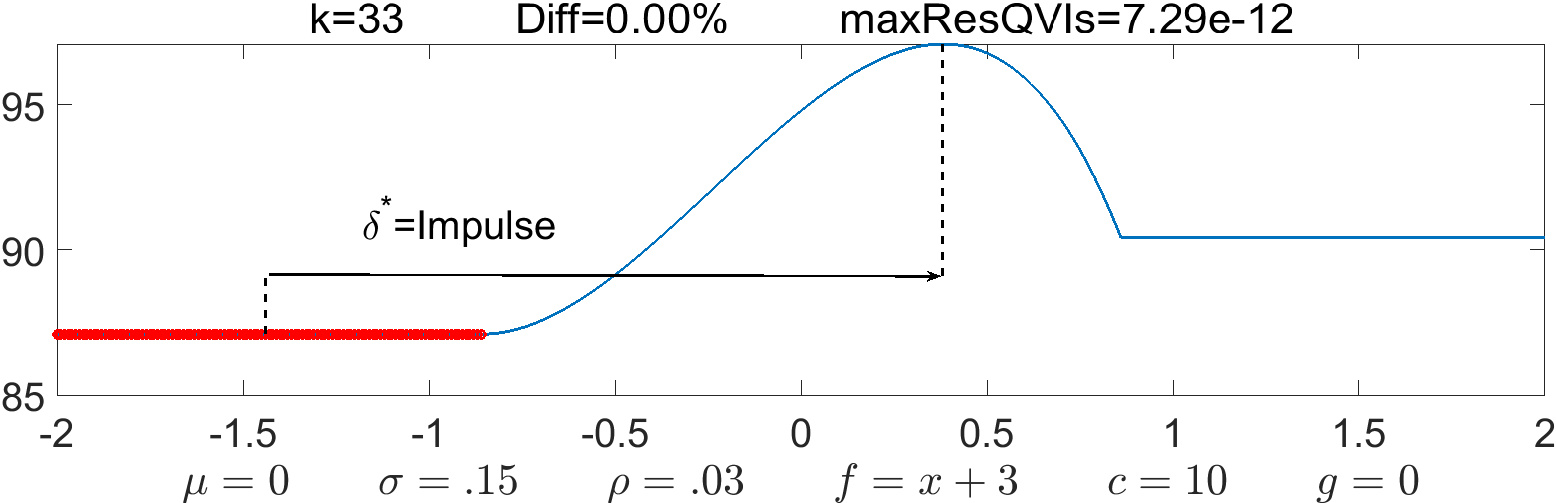}
\end{figure}
\begin{figure}[H]
\hspace*{.2cm}
	\includegraphics[scale=.36]{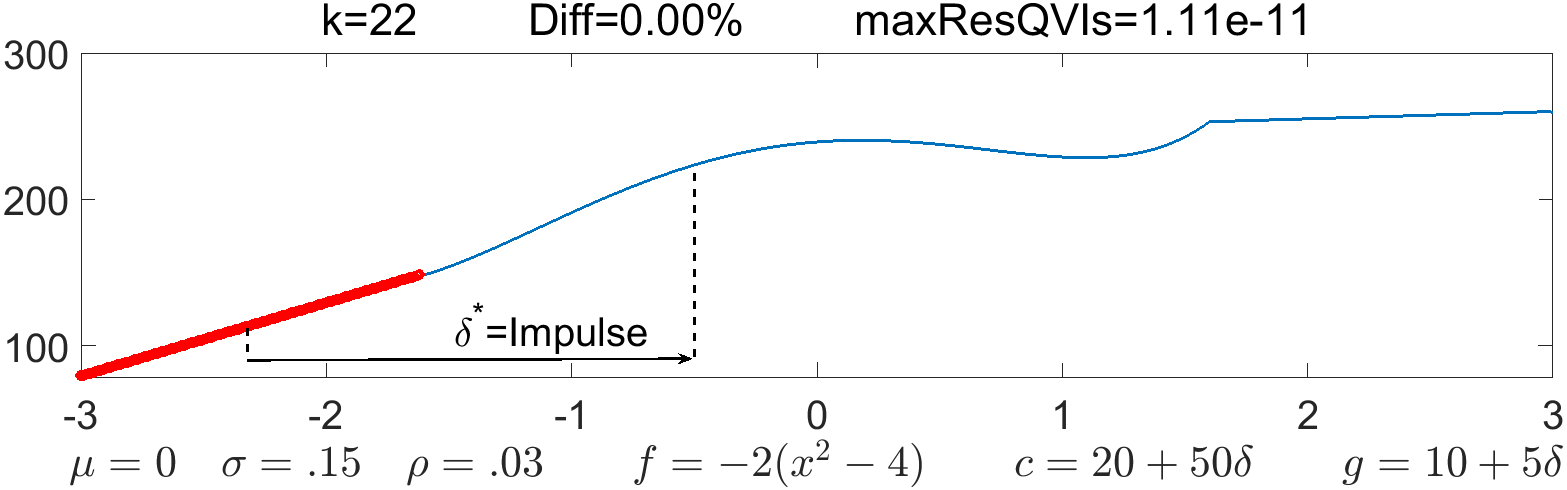}
\end{figure}
\begin{figure}[H]
\hspace*{.2cm}
	\includegraphics[scale=.36]{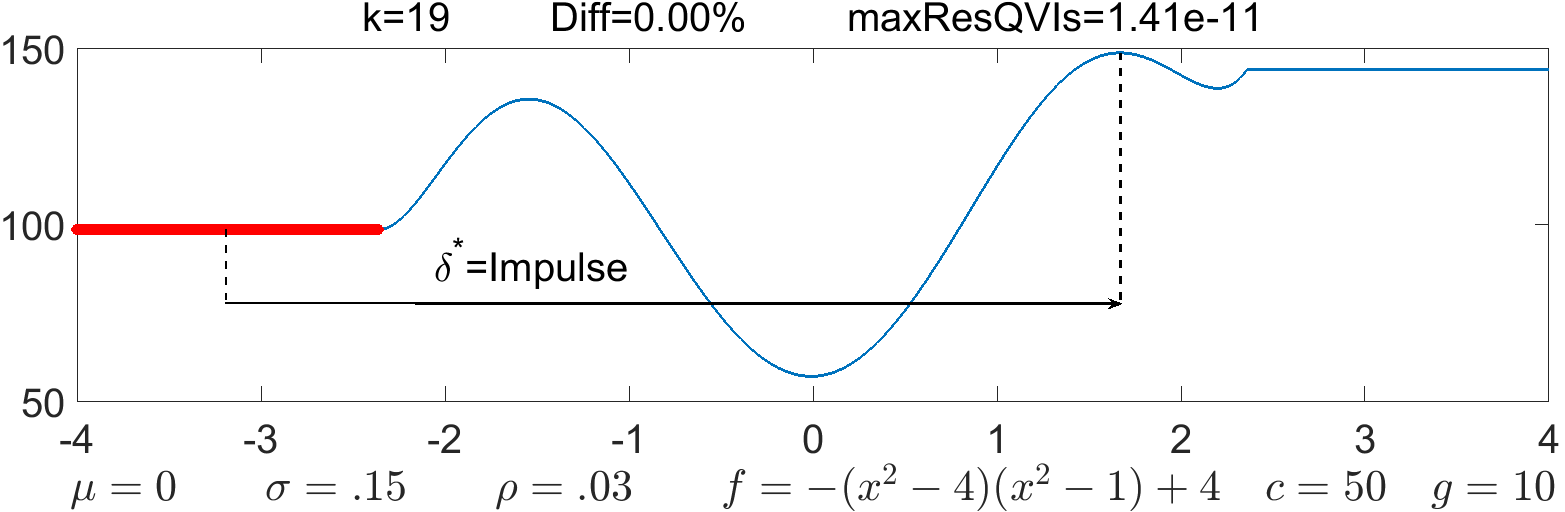}
\end{figure}
The following is an example in which the accuracy stagnates. At that point, the iterates start going back and forth between a few values. Although we cannot guarantee that we are close to a solution of (\ref{dQVIs}), the results seem quite convincing, with both Diff and maxResQVIs reasonably low. In fact, simply halving the step size to $h=0.005$ produces a substantial improvement of Diff=9.16e-11\% and maxResQVIs=9.19e-11 in $k=$33 iterations.
\begin{figure}[H]
\hspace*{.4cm}
	\includegraphics[width=.9\linewidth]{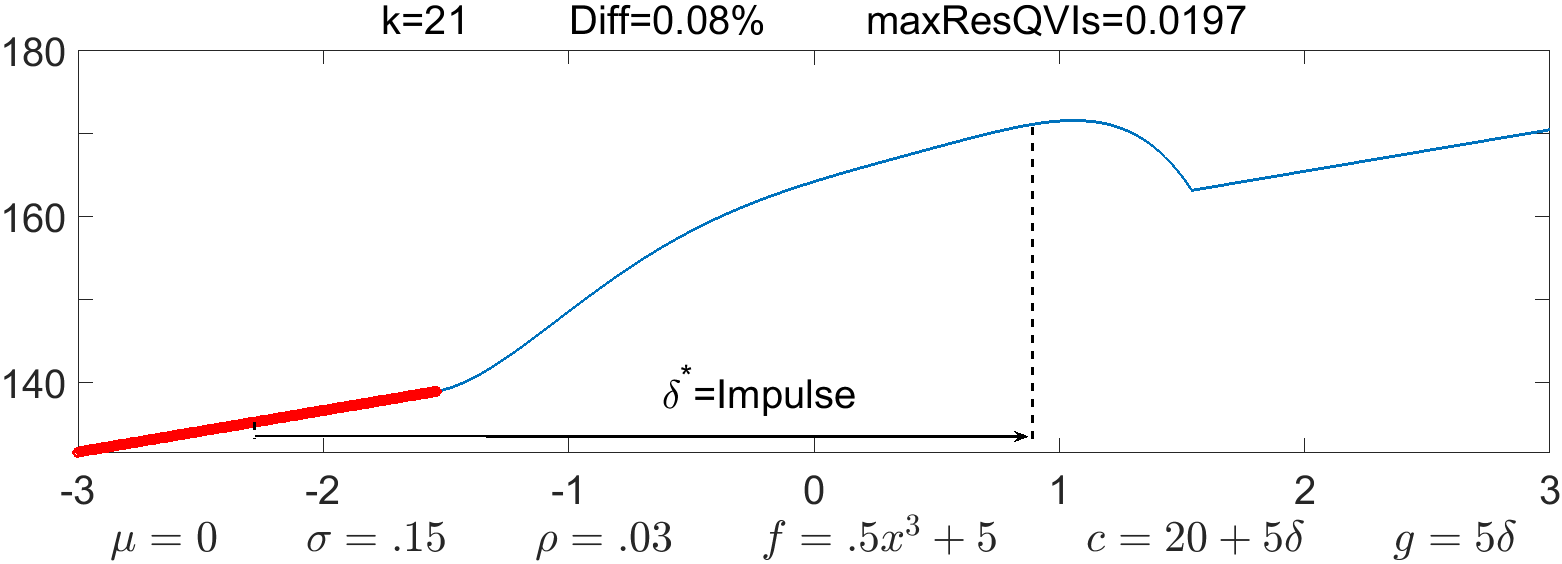}
\end{figure}
The previous games have a state variable evolving as a scaled Brownian motion. We now move on to a mean reverting OU process with zero long term mean. (Recall that any other value can be handled simply by shifting the game.) In general, the experiments with these dynamics converged exactly in a very small number of iterations.
\begin{figure}[H]
\hspace*{.2cm}
	\includegraphics[scale=.36]{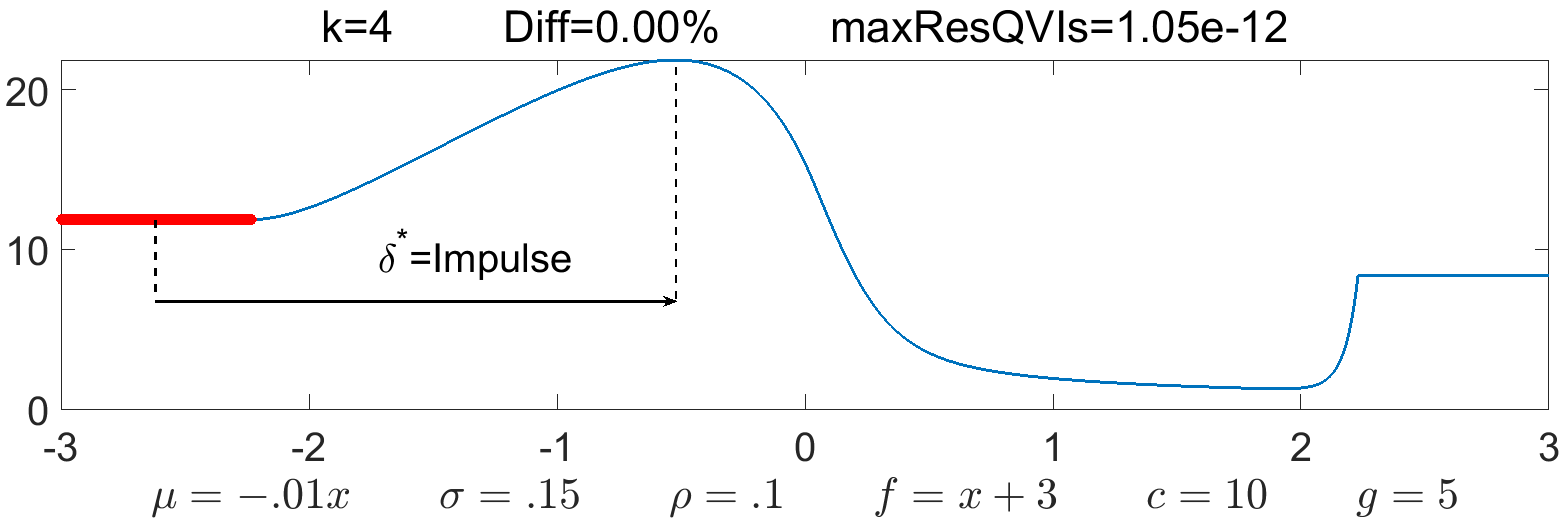}
\end{figure}
\begin{figure}[H]
\hspace*{.2cm}
	\includegraphics[scale=.36]{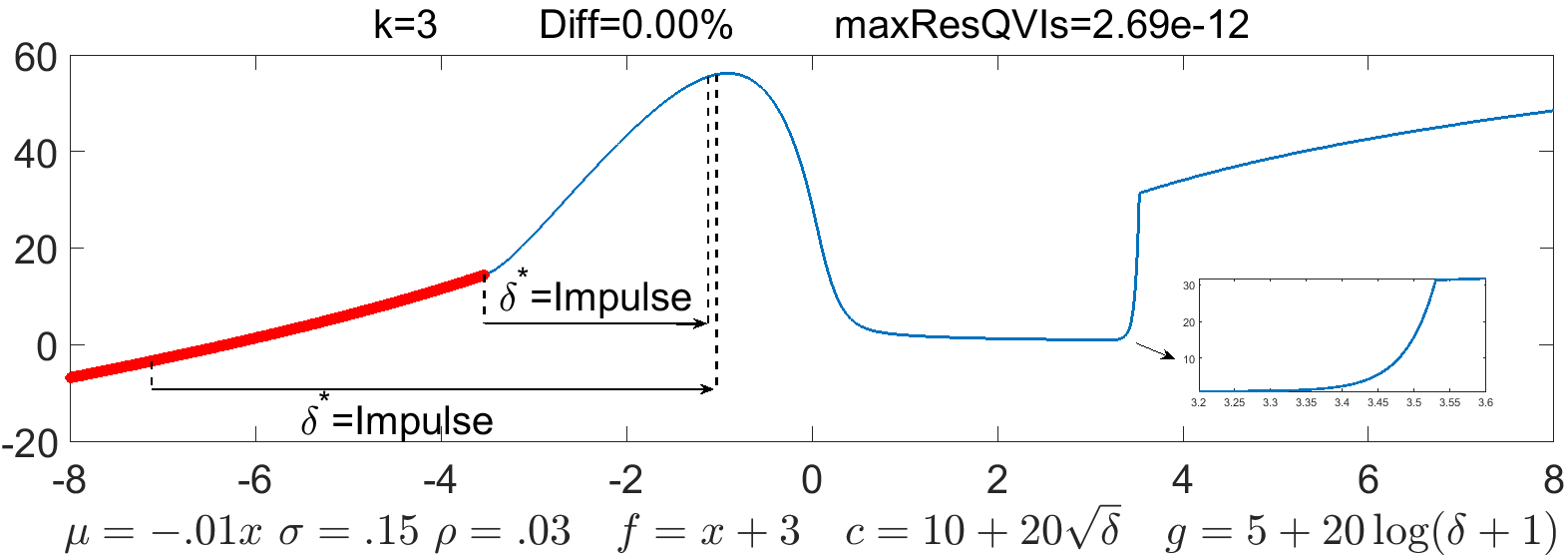}
\end{figure}
Note how all the games treated in this section exhibit a typical feature known to hold for simpler symmetric games \cite{ABCCV,ABMZZ,BCG}: the equilibrium payoff of the player is only $C^1$ at the border of the intervention region $\partial\mathcal I^*=\{\mathcal AV-\rho V+f = \mathcal MV-V \}\cap\mathbb R_{<0}$, and only continuous at the border of the opponent's intervention region $-\partial \mathcal I^*$. In floating point arithmetic, the former makes the discrete approximation of $\partial\mathcal I^*$ particularly elusive, while the latter can lead to high errors when close to $-\partial\mathcal I^*$. As a consequence, Subroutine \ref{solve_one_player} (or any equivalent) will often misplace a few grid nodes between intervention and continuation regions, which will in turn make the residual resQVIs ``spike" on the opponent's side. Thus, a large value of maxResQVIs can at times be misleading, and further inspection of the pointwise residuals is advisable. 

As a matter of fact, halving the step size to $h=0.005$ in the last example results in exact convergence with terminal maxResQVIs$=2510$, but the residuals on all grid nodes other than the ``border" of the opponent's intervention region and a contiguous one are less than 1.43e-11. This is an extreme example propitiated by the almost vertical shape of the solution close to such border. Thus, while it is useful in practice to consider a stopping criteria for Algorithm \ref{sym_algo} based on maxResQVIs, this phenomenon needs to be minded. 

The last example also shows how impulses at a NE can lead to different endpoints depending on the state of the process. This is often the case when costs are nonlinear. In fact:

\begin{lemma}
\label{nonunique_endpoint}
Let $(\mathcal I^*,\delta^*)$ and $V$ be a symmetric NE and equilibrium payoff as in Corollary \ref{coro_sym_QVIs}. Assume that $\mathcal Z(x)=[0,+\infty)$ for $x<0$ and $c=c(x,\delta)\in C^2\big(\mathbb R\times (0,+\infty)\big)$, and consider the re-parametrization $\overline c=\overline c(x,y) \defeq c(x,y-x)$. Suppose that $y^*\defeq x+\delta^*(x)$ is constant for all $x\in{(\mathcal I^*)}^\circ$.\footnote{$A^\circ$ denotes the interior of the set $A\subseteq \mathbb R$.} Then $\overline c_{xy}(x,y^*)\equiv 0$ on ${(\mathcal I^*)}^\circ$. 
\end{lemma}

The result is immediate since $\mathcal MV(x)=\sup_{y>x}\{V(y)-\overline c(x,y)\} = V(y^*)-\overline c(x,y^*)$ on $(\mathcal I^*)^\circ$ and $V\in C^1\big(-\mathcal C^*)$. ($y^*\notin -\mathcal I^*$ or there would be infinite simultaneous impulses.) While the sufficient condition $\overline c_{xy}(x,y^*)\equiv 0$ for some $y^*$ is verified for linear costs, it is not in general and certainly not for $c(x,\delta)=10+20\sqrt\delta$ as in the last example.

\subsection{Convergence to analytical solution with refining grids}
\label{cv_analytical_sol}
A convergence analysis from discrete to analytical solution with refining grids is outside the scope of this paper, and seems to be far too challenging when a viscosity solutions framework is yet to be developed. Instead, we present here a numerical validation using the solutions of the linear and cash management games. We focus first and foremost in the former, as it has an almost fully analytical solution, with only one parameter to be found numerically as opposed to four for the latter. The structure of the game is defined with parameter values used in \cite{ABCCV} (also in \cite{ABMZZ}). To minimize rounding errors from floating point arithmetic, we proceed as in \cite{AF} considering grids made up entirely of machine numbers.\footnote{By \emph{machine numbers} we mean those that can be represented exactly in IEEE standard floating-point number system.} The results are displayed in Table \ref{table}. 

For each step size, Algorithm \ref{sym_algo} either converged or was terminated upon stagnation. Regardless, we can see the errors made when approximating the analytical solution are quite satisfactory in all cases, and overall decrease as $h\to 0$. Moreover, we see once again how the ``spiking" of the residual can be misleading (the highest value was always at the ``border" of the opponent's intervention region). 

\begin{SCtable}[][h]
	\begin{tabular}{llll}
		h  &   $\%\mbox{error}$  &  Its.  &  maxResQVIs \\ 
		\noalign{\smallskip}\hline\noalign{\smallskip}
		1  &      6.67$\%$       &   17  &\quad    8.88e-16 \\
	 1/2 &      8.33$\%$       &   13  &\quad    5.33e-15 \\
	 1/4 &      0.23$\%$       &    4  &\quad     13.2    \\
	 1/8 &      0.21$\%$       &    8  &\quad     15.1    \\
	1/16 &      0.16$\%$       &    8  &\quad      30     \\
	1/32 &      0.07$\%$       &   21  &\quad     21.2    \\
	1/64 &      0.0043$\%$     &   37  &\quad     0.343   \\
		\noalign{\smallskip}\hline
	\end{tabular}
	\caption{Convergence to analytical solution when refining equispaced symmetric grid with step size $h$ and endpoint $x_N=4$. Game: $\mu=0,\ \sigma=.15,\ \rho=.02,\ f=x+3,\ c=100+15\delta,\ g=15\delta$. \%error $\defeq \|(v-V)/V\|_{\infty}$, with $V$ exact solution and $v$ discrete approximation after Its iterations. $\|\cdot\|_\infty$ is computed over the grid.}
	\label{table} 
\end{SCtable}
An exact symmetric NE for this game (up to five significant figures) is given by the intervention region $\mathcal I^*=(-\infty,-2.8238]$ and impulse function $\delta^*(x)=1.5243-x$, while the approximation given by Algorithm \ref{sym_algo} with $h=1/64$ is $\big((-\infty, -2.8125],1.5313 - x\big)$, with absolute errors on the parameters of no more than the step size.

The cash management game \cite{BCG} with unidirectional impulses can be embedded in our framework by reducing its dimension with the change of variables $x=x_1-x_2$, changing minimization by maximization and relabelling the players. With the parameter values of \cite[Fig.1b]{BCG}, it translates into the game: $\mu=0,\ \sigma=1,\ \rho=.5,\ f=-|x|,\ c=3+\delta,\ g=-1$. The authors found numerically a symmetric NE approximately equal to $\big((-\infty, -5.658], 0.686 - x\big)$, while Algorithm \ref{sym_algo} with $x_N=8$ and $h=1/64$ gives $\big((-\infty, -5.6563], 0.6875 - x\big)$. The absolute difference on the parameters is once again below the grid step size.

\subsection{Games without Nash equilibria}
\label{s:games_without_NE}
It is natural to ask how Algorithm \ref{sym_algo} behaves on games without symmetric NEs, and whether anything can be inferred from the results. For the linear game, two cases without NEs (symmetric or not) are addressed in \cite{ABCCV}: ``no fixed cost" and ``gain greater than cost". Both of them yield degenerate ``equilibria" where the players perform infinite simultaneous interventions. When tested for several parameters, Algorithm \ref{sym_algo} converged in finitely many iterations (although rather slowly) and yielded the exact same type of ``equilibria".\footnote{There were also cases of stagnation, improved by refining the grid as in Section \ref{s:games_on_fixed_grid}.} For a fine enough grid, the previous can be identified heuristically by some node in the intervention region that would be shifted to its symmetric one (\textit{infinite alternated interventions}), or to its immediate successor over and over again, until reaching the continuation region (\textit{infinite one-sided interventions}).\footnote{More precisely, due to our choice of $Z(x)$ a node can be shifted at most to that immediately preceding its symmetric one.}

In the first case, we recovered the limit ``equilibrium payoffs" of \cite[Props.4.10,4.11]{ABCCV}. Intuitively, the players in this game take advantage of free interventions, whether by no cost or perfect compensation, in order to shift the process as desired. Note that when $c=g$, the impulses that maximimize the net payoff are not unique.

In the second case, grid refinements showed the discrete payoffs to diverge towards infinity at every point. This is again consistent with the theory: each player forces the other one to act, producing a positive instantaneous profit. Iterating this procedure infinitely often leads to infinite payoffs for every state.

Tested games in which Algorithm \ref{sym_algo} failed to converge (and not due to stagnation nor a poor choice of the grid extension) were characterized by iterates reaching a cycle, typically with high values of Diff and maxResQVIs regardless of the grid step size. In many cases, the cycles would visit at least one payoff inducing infinite simultaneous impulses. While this might, potentially and heuristically, be indicative of the game admitting no symmetric NE, there is not much more than can be said at this stage.

\subsection{Beyond the Verification Theorem}
\label{s:beyond_verif_theo}
We now present two games in which the solution of the discrete QVIs system (\ref{dQVIs}) found with Algorithm \ref{sym_algo} (by exact convergence) does not comply with the continuity and smoothness assumptions of Corollary \ref{coro_sym_QVIs}. However, the results are sensible enough to heuristically argue they may correspond to NEs beyond the scope of the Verification Theorem. In both cases $h=0.01$. Finer grids yielded the same qualitative results. 

The first game considers costs convex in the impulse. When far enough from her continuation region, it is cheaper for the player to apply several (finitely many) simultaneous impulses instead of one, to reach the state she wishes to (cf. Remark \ref{r:concave_costs}). 
In this game, the optimal impulse $\delta^*$ becomes discontinuous, and its discontinuity points are those in the intervention region where the equilibrium payoff is non-differentiable. These, in turn, translate into discontinuities on the opponent's intervention region (one degree of smoothness less, as with the border of the regions). 
\begin{figure}[H]
\hspace*{.2cm}
	\includegraphics[scale=.36]{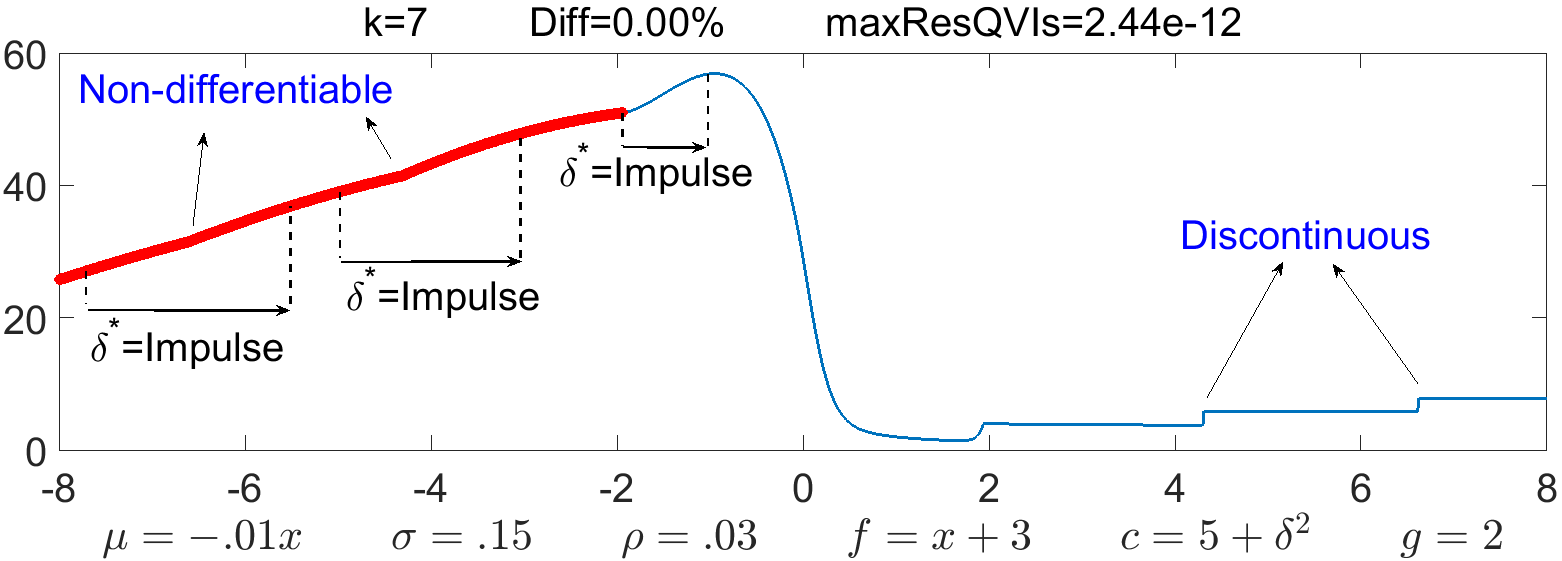}
\end{figure}

The second game considers linear gains, quadratic running payoffs and costs concave on the impulse. The latter makes the player shift the process towards her continuation region (cf. Remark \ref{r:concave_costs}). However, when far enough from the border, instead of shifting the process directly to her ``preferred area", the player chooses to pay a bit more to force her opponent's intervention, inducing a gain and letting the latter pay for the final move. Once again, this causes $\delta^*$ to be discontinuous and leads to a non-differentiable (resp. discontinuous) point for the equilibrium payoff in the intervention (resp. opponent's intervention) region.
\begin{figure}[H]
\hspace*{.2cm}
	\includegraphics[scale=.36]{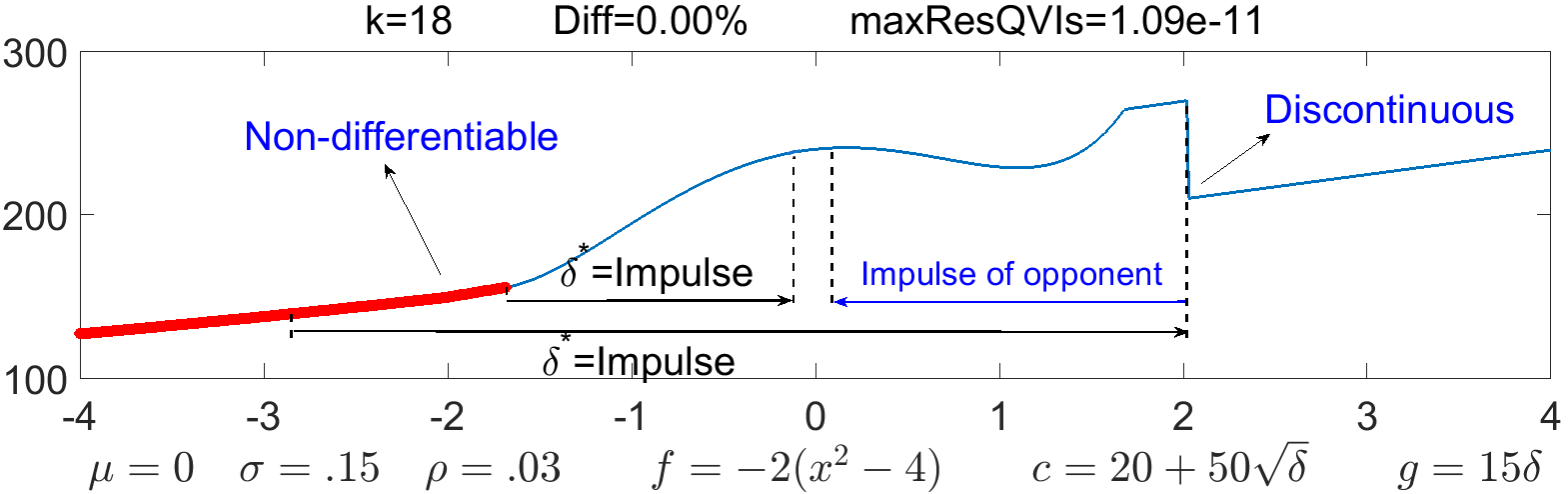}
\end{figure}

Under the previous reasoning, one could intuitively guess that setting $g=0$ in this game should remove the main incentive the player has to force her opponent to act. This is in fact the case, as shown below. As a result, $\delta^*$ becomes continuous and the equilibrium payoff falls back into the domain of the Verification Theorem. 
\begin{figure}[H]
\hspace*{.2cm}
	\includegraphics[scale=.36]{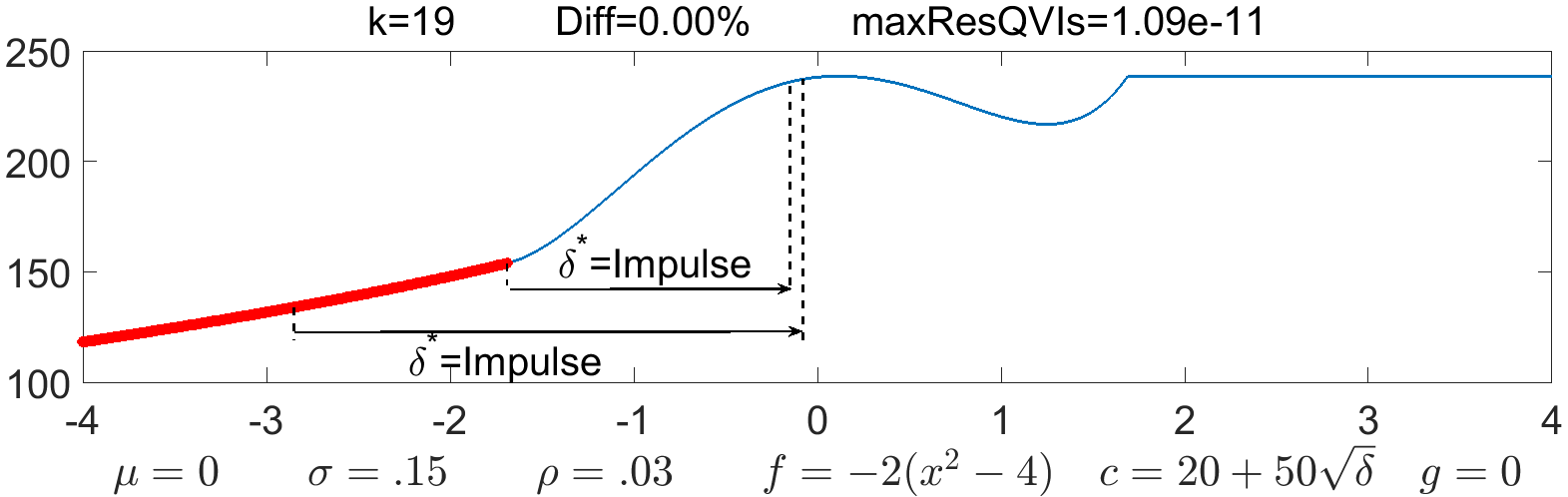}
\end{figure}
 
We remark that, should the previous solutions correspond indeed to NEs, then the alternative semi-analytical approach of \cite{FK} could not have produced them either, as that method can only yield continuous equilibrium payoffs. 
	
%%%%%%%%%%%%%%%%%%%%%%%%%%%%%%%%%%%%%%%%%%%  CONCLUSIONS %%%%%%%%%%%%%%%%%%%%%%%%%%%%%%%%%%%%%%%%%%%%%%%%%%%%%%%%%%%%%
\section{Concluding remarks}
\label{s:conclusions}
This paper presents a fixed-point policy-iteration-type algorithm to solve systems of quasi-variational inequalities resulting from a Verification Theorem for symmetric nonzero-sum stochastic impulse games. In this context, the method substantially improves the only algorithm available in the literature while providing the convergence analysis that we were missing. Graph-theoretic assumptions relating to weakly chained diagonally dominant matrices, which naturally parallel the admissibility of the players strategies, allow to prove properties of contractiveness, boundedness of iterates and convergence to solutions. A result of theoretical interest giving sufficient conditions for convergence is also proved. Additionally, a novel provably convergent impulse control solver is provided.

Equilibrium payoffs and Nash equilibria of games too challenging for the available analytical approaches are computed with high precision on a discrete setting. Numerical validation with analytical solutions is performed when possible, with reassuring results, but it is noted that grid refinements may be needed at times to overcome stagnation. Thus, formalising the approximating properties of the discrete solutions as well as deriving stronger convergence results for the algorithm may need a viscosity solutions framework currently missing in the theory. This is further substantiated by the irregularity of the solutions, particularly those found which escape the available theoretical results. This motivates further research while providing a tool that can effectively be used to gain insight into these very challenging problems.

%%%%%%%%%%%%%%%%%%%%%%%%%%%%%%%%%%%%%%%%%%%%%%% APPENDIX %%%%%%%%%%%%%%%%%%%%%%%%%%%%%%%%%%%%%%%%%%%%%%%%%%%%%%%%%%%%%%
\appendix
\section{Matrix and graph-theoretic definitions and results}
\label{appendix:matrices}
\setcounter{subsection}{1}

For the reader's convenience, this appendix summarizes some important algebraic and graph-theoretic definitions and results used throughout the paper. More details can be found in the references given below. Henceforth, $A\in\mathbb R^{N\times N}$ is a real matrix, $\operatorname{Id}\in\mathbb R^{N\times N}$ is the identity, $\rho(\cdot)$ denotes the spectral radius and $\mathbb R^{N},\mathbb R^{N\times N}$ are equipped with the elementwise order. We talk about rows and ``states" interchangeably.
\begin{definitions}
\label{matrix_definitions}
\begin{enumerate}[label=(D\arabic*)]
\item $A$ is a \textit{Z-matrix} if it has nonpositive off-diagonal elements.
\item $A$ is an \textit{$L$-matrix} (resp. \textit{$L_0$-matrix}) if it is a Z-matrix with positive (resp. nonnegative) diagonal elements.
\item $A$ is an \textit{$M$-matrix} if $A=s\operatorname{Id} - B$ for some matrix $B\geq 0$ and scalar $s\geq \rho(B)$.
\item $A$ is \textit{monotone} if it is nonsingular and $A^{-1}\geq 0$. Equivalently, $A$ is monotone if $Ax\geq 0$ implies $x\geq 0$ for any $x\in\mathbb R^N$. 
\item The $i$-th row of $A$ is \textit{weakly diagonally dominant (WDD)} (resp. \textit{strictly diagonally dominant or SDD}) if $|A_{ii}|\geq\sum_{j\neq i}|A_{ij}|$ (resp. $>$).
\item $A$ is \textit{WDD} (resp. \textit{SDD}) if every row of $A$ is WDD (resp. SDD). 
\item The \textit{directed graph of }$A$ is the pair graph$A\defeq(V,E)$, where $V\defeq\{1,\dots,N\}$ is the set of \textit{vertexes} and $E\subseteq V\times V$ is the set of \textit{edges}, such that $(i,j)\in E$ iff $A_{ij}\neq 0$.
\item A \textit{walk} $\omega$ in graph$A=(V,E)$ from vertex $i$ to vertex $j$ is a nonempty finite sequence $(i,i_1),(i_1,i_2),\dots,(i_{k-1},j)\subseteq E$, which we denote by $\omega:i\to i_1\to\dots\to i_{k-1}\to j$. $|\omega|\defeq k$ is called the \textit{length} of the walk $\omega$.
\item $A$ is \textit{weakly chained diagonally dominant (WCDD)} if it is WDD and for each WDD row of $A$ there is a walk in graph$A$ to an SDD row (identifying vertexes and rows).
\item $A$ is (right) \textit{substochastic or sub-Markov} (resp. \textit{stochastic or Markov}) if $A\geq 0$ and each row sums at most one (resp. exactly one). Equivalently, $A$ is substochastic if $A\geq 0$ and $\|A\|_\infty\leq 1$. (Recall that $\|A\|_\infty$ is the maximum row-sum of absolute values.)
\item If $A$ is a WDD (resp. substochastic) matrix, its set of ``non-trouble states" (or rows) is $J[A]\defeq \{i:\mbox{the }i\mbox{-th row of }A\mbox{ is SDD}\}$ (resp. $\hat J[A]\defeq \{i:\sum_j A_{ij}<1\}$). For each $i$, we write $P_i[A]\defeq\{\mbox{walks in graph}A\mbox{ from }i\mbox{ to some }j\in J[A]\}$ (resp. we define $\widehat P_i[A]$ analogously). The \textit{index of connectivity} (resp. \textit{index of contraction}) of $A$ \cite{A} is 
$$
\mbox{con}A\defeq\left(\sup_{i\notin J[A]}\left\{ \inf_{\omega\in P_i[A]}|\omega|\right\}\right)^+\mbox{ (resp. we define $\widehat{\mbox{con}}A$ analogously).}\footnote{$(\cdot)^+$ denotes positive part, $\inf\emptyset=+\infty$ and $\sup\emptyset=-\infty$. The index is the least length that needs to be walked on graph$A$ to reach the non-trouble states when starting from an arbitrary trouble one.}
$$
\end{enumerate}
\end{definitions} 

It is clear that SDD$\implies$WCDD$\implies$WDD, and by definition, L-matrix$\implies\mbox{L}_0$-matrix$\implies$Z-matrix. Also by definition, if $A$ is WDD then: $A$ is WCDD $\iff$ con$A<+\infty$.

\begin{proposition}(e.g., \cite{Sh} or \cite[Lem.3.2]{AF})
\label{WCDD_is_nonsing}
Any WCDD matrix is nonsingular.
\end{proposition}

\begin{proposition}(e.g., \cite[Prop.2.15 and 2.17]{A}) \label{M_monotone_equiv}
\label{M-matrix_characterization}

Nonsingular M-matrix $\iff$ monotone L-matrix $\iff$ monotone Z-matrix.
\end{proposition}

\begin{theorem}(e.g., \cite[Thm.2.24]{A})
\label{characterization_theorem}
WCDD $\mbox{L}_0$-matrix $\iff$ WDD nonsingular M-matrix.\footnote{\cite[Thm.2.24]{A} is formulated in terms of L-matrices instead. However, it is trivial to see that: WCDD $\mbox{L}_0$-matrix $\iff$ WCDD L-matrix.}
\end{theorem}

\begin{proposition} (see proof of \cite[Lem.2.22]{A})
\label{substochastic_WDD_link}
$A$ is substochastic if and only if $\operatorname{Id}-A$ is a WDD $\mbox{L}_0$-matrix and $A$ has non-negative diagonal elements. In such case, $\hat J[A]=J[\operatorname{Id}-A]$, they have the same directed graphs (except possibly for self-loops $i\to i$) and $\widehat{\emph{\mbox{con}}}A=\emph{\mbox{con}}[\operatorname{Id}-A]$.
\end{proposition}

For the following theorem, recall the characterization of the spectral radius $\rho(A)=\inf\{\|A\|:\|\cdot\|\mbox{ is a matrix norm}\}$ and Gelfand's formula $\rho(A)=\lim_{n\to+\infty}\|A^n\|^{1/n}$, for any matrix norm $\|\cdot\|$. Note also that if $A$ is substochastic, then $A^n$ is also substochastic for any $n\in\mathbb N_0$, $\|A^n\|_{\infty}\leq 1$ and $\rho(A)\leq 1$. 
\begin{theorem}(\cite[Thm.2.5 and Cor.2.6]{A})
\label{contraction}
Suppose $A$ is substochastic. Then 
$$\widehat{\emph{\mbox{con}}}A=\inf\{n\in\mathbb N_0:\|A^{n+1}\|_\infty<1\}.$$ In particular, $\widehat{\emph{\mbox{con}}}A<+\infty$ if and only if $\rho(A)<1$.
\end{theorem}

The indices of contraction and connectivity can be generalized in a natural way to sequences $(A_k)\subseteq\mathbb R^{N\times N}$ by considering walks $i_1\to i_2\to\dots$ such that $i_k\to i_{k+1}$ is an edge in graph$A_k$ (see \cite[App.B]{A} for more details). Theorem \ref{contraction} extends in the following way:

\begin{theorem}(\cite[Thm.B.2]{A})
\label{sequence_contraction}
Suppose $(A_k)$ are substochastic matrices and consider the sequence of products $(B_k)$, where $B_k\defeq A_1\dots A_k$. Then,
$$\widehat{\emph{\mbox{con}}}\big[(A_k)\big]=\inf\{k\in\mathbb N_0:\|B_{k+1}\|_\infty<1\}.$$
Equivalently, $1=\|B_0\|_{\infty}=\dots=\|B_\alpha\|_{\infty}>\|B_{\alpha+1}\|_{\infty}\geq \dots$, for $\alpha=\widehat{\emph{\mbox{con}}}\big[(A_k)\big]$.\footnote{The equivalence is due to $(A_k)$ being substochastic matrices, which implies that $\|B_k\|_{\infty}\leq 1$ for all $k$. Note that the case $\alpha=+\infty$ reads $\|B_k\|_{\infty}=1$ for all $k$. } 
\end{theorem}

\begin{proof}[\textbf{Proof of Lemma \emph{\ref{coefficients_properties}}}]
\phantomsection
\label{proof:lemma}

For briefness, we omit the dependence on $\varphi,\overline\varphi$ from the notation.

It follows from Proposition \ref{FPPI-like_result} that $\mathbb A^{-1}\mathbb B\geq 0$, since $\mathbb B\geq 0$ and $\mathbb A$ is monotone. To see that its rows sum up to one, let $\boldsymbol{1}\in\mathbb R^\grid$ be the vector of ones. It is easy to check using $\mathbb A$ and $\mathbb B$ explicit expressions and \ref{A1}--\ref{A2}, that $\mathbb A\mathbb A^{-1}\mathbb B\boldsymbol{1}=\mathbb B\boldsymbol{1}\leq \mathbb A\boldsymbol{1}$, which implies $\mathbb A^{-1}\mathbb B\boldsymbol{1}\leq \boldsymbol{1}$. This proves that $\mathbb A^{-1}\mathbb B$ is substochastic.

The WDD $\mbox{L}_0$-property of $\mathbb A-\mathbb B$ is due to \ref{A1}. To verify this, note first that by straightforward computations, if a matrix $A_0$ has the previous property then $SA_0S$ must also have it. Then simply observe that the $i$-th row of $\mathbb A-\mathbb B$ is equal to the $i$-th row of $\operatorname{Id}-B(\overline\delta)$ (if $x_i\in\overline I$), or $-L$ (if $x_i\in{\overline I}^c\cap (-I)^c$) or $\operatorname{Id}-SB(\delta)S=S(\operatorname{Id}-B(\delta))S$ (if $x_i\in -I$), all of which are WDD $\mbox{L}_0$-matrices by \ref{A1}. 

We prove now the last statement, $\widehat{\mbox{con}}\big[\mathbb A^{-1}\mathbb B\big]\leq\mbox{con}\big[\mathbb A-\mathbb B\big]$. Note that by Proposition \ref{substochastic_WDD_link}, $\mathbb A^{-1}\mathbb B$ and $\operatorname{Id}-\mathbb A^{-1}\mathbb B$ share the same sets of non-trouble rows, directed graphs (except possibly for self loops) and $\widehat{\mbox{con}}\big[\mathbb A^{-1}\mathbb B\big]=\mbox{con}\big[\operatorname{Id}-\mathbb A^{-1}\mathbb B\big]$. Thus, we want see that for any non-SDD row of $\operatorname{Id}-\mathbb A^{-1}\mathbb B$, if an SDD row of $\mathbb A-\mathbb B$ can be reached by a walk in the graph of the latter, then there must be a shorter walk to an SDD row in the graph of the former. We proceed by showing several results:
\bigskip

\emph{(i) There is a walk from $x_i$ to $x_j$ in \emph{graph}$\mathbb A$ if and only if $\mathbb A^{-1}_{ij}>0$.\footnote{The reciprocal implication is shown merely for the sake of completeness.} In particular, $\mathbb A^{-1}$ has strictly positive diagonal elements.}\emph{(More in general, both statements are true for any WCDD $\mbox{L}_0$-matrix, with the same proof)}

Recall first that $\mathbb A$ is a WCDD $\mbox{L}_0$-matrix and so it must have positive diagonal elements, or there would be a null row disconnected from every SDD row. Thus, $x_i\to x_i$ is always in graph$\mathbb A$ and we must show that $\mathbb A^{-1}_{ii}>0$. If $\mathbb A^{-1}_{ii}=0$, then by monotonicity and $\mbox{L}_0$-property, it would be $1=\big[\mathbb A\mathbb A^{-1}\big]_{ii}=\sum_{j\neq i}\mathbb A_{ij}\mathbb A^{-1}_{ji}\leq 0$.

Suppose now $i\neq j$ and that $x_i$ is connected to $x_j$ by a walk in graph$\mathbb A$. We proceed by induction in the length of the walk. If the length is one, suppose by contradiction that $\mathbb A^{-1}_{ij}=0$. Then $0=\big[\mathbb A\mathbb A^{-1}\big]_{ij}=\sum_{k\neq i}\mathbb A_{ik}\mathbb A^{-1}_{kj}\leq \mathbb A_{ij}\mathbb A^{-1}_{jj}<0$ by assumption and the diagonal case. For an arbitrary walk $x_i\to y_1\to\dots\to y_n=x_j$, assume without loss of generality that it has no closed subwalks. Then $y_1=x_h$ for some $h\neq i,j$. Once again if $\mathbb A^{-1}_{ij}=0$, this would give $0=\big[\mathbb A\mathbb A^{-1}\big]_{ij}=\sum_{k\neq i}\mathbb A_{ik}\mathbb A^{-1}_{kj}\leq \mathbb A_{ih}\mathbb A^{-1}_{hj}<0$ by induction.

For the reciprocal, assume now that $j\neq i$ and $\mathbb A^{-1}_{ij}>0$. By the adjoint formula of the inverse and Leibniz (or permutations) formula of the determinant, there must exist some permutation $\sigma:\{-N,\dots,N\}\backslash\{j\}\to\{-N,\dots,N\}\backslash\{i\}$ such that $\prod_{k\neq j} \mathbb A^{-1}_{k\sigma(k)}\neq 0$. Hence, $x_i\to x_{\sigma(i)}\to x_{\sigma^2(i)}\to...\to x_j$ is a walk in graph$\mathbb A$.
\bigskip

\emph{(ii) $\operatorname{Id} - \mathbb A^{-1}\mathbb B$ has no less SDD rows than $\mathbb A-\mathbb B$. Further, $J\big[\mathbb A-\mathbb B\big]\subseteq J\big[\operatorname{Id} - \mathbb A^{-1}\mathbb B\big]$:}

\noindent Consider some $x_i\in J\big[\operatorname{Id} - \mathbb A^{-1}\mathbb B\big]^c$, i.e., $\sum_j [\operatorname{Id} - \mathbb A^{-1}\mathbb B]_{ij}=0$. We want to see that $\sum_j[\mathbb A-\mathbb B]_{ij}=0$.  We have $0 = \sum_j\big[\operatorname{Id}-\mathbb A^{-1}\mathbb B\big]_{ij}=\sum_j\big[\mathbb A^{-1}(\mathbb A-\mathbb B)\big]_{ij} = \sum_k \mathbb A^{-1}_{ik}\sum_{j}\big[\mathbb A-\mathbb B\big]_{kj}$. Since $\mathbb A^{-1}_{ik}\mbox{ and }\sum_{j}\big[\mathbb A-\mathbb B\big]_{kj}$ are non-negative for all $k$ (monotonicity and $L_0$ WDD property, respectively), one of the two must be zero for each $k$. But $\mathbb A^{-1}_{ii}>0$ by \emph{(i)}, giving what we wanted. 
\bigskip

\emph{(iii) If $x_i\in -I$ and there is an edge $x_i\to x_j$ in \emph{graph}$\big(\mathbb A-\mathbb B\big)$ for some $j\neq i$, then this is also an edge in \emph{graph}$\big(\operatorname{Id} - \mathbb A^{-1}\mathbb B\big)$:}

$
\big[\operatorname{Id} - \mathbb A^{-1}\mathbb B\big]_{ij}=-\big[\mathbb A^{-1}\mathbb B\big]_{ij}=-\sum_k \mathbb A^{-1}_{ik}\mathbb B_{kj}\leq 
-\mathbb A^{-1}_{ii}\mathbb B_{ij}=\mathbb A^{-1}_{ii}\big[\mathbb A - \mathbb B\big]_{ij}<0,
$
where the last equality is due to $x_i\in -I$ and $j\neq i$, and the strict inequality is by assumption and \emph{(i)}.\bigskip

\emph{(iv) If there is a walk in \emph{graph}$\big(\mathbb A-\mathbb B\big)$, $x_i=y_0\to\dots\to y_n\to x_j\in J\big[\mathbb A-\mathbb B\big]$, with $y_m\in\overline I$ for all $m$, then $x_i\in J\big[\operatorname{Id} - \mathbb A^{-1}\mathbb B\big]$:}

Since $\mathbb A-\mathbb B$ and $\mathbb A$ agree for every row indexed in $\overline I$, $\mathbb A^{-1}_{ij}>0$ by \emph{(i)}. It follows that $\sum_k \big[\operatorname{Id} - \mathbb A^{-1}\mathbb B\big]_{ik}=\sum_k \big[\mathbb A^{-1}(\mathbb A - \mathbb B)\big]_{ik}=\sum_h \mathbb A^{-1}_{ih} \sum_k \big[\mathbb A- \mathbb B\big]_{hk}\geq \mathbb A^{-1}_{ij} \sum_k \big[\mathbb A- \mathbb B\big]_{jk}>0$. \bigskip

\emph{(v) If there is a walk in \emph{graph}$\big(\mathbb A-\mathbb B\big)$, $x_i=y_0\to\dots\to y_n\to x_j\to x_h$, with $y_m\in\overline I$ for all $m$, $x_j\in -I$ and $h\neq i,j$, then $x_i\to x_h$ is an edge in \emph{graph}$\big(\operatorname{Id} - \mathbb A^{-1}\mathbb B\big)$:}

Using that $\mathbb A-\mathbb B\leq\mathbb A$, with equality for rows in $\overline I$, and $\mathbb A^{-1}_{ij}\big[\mathbb A - \mathbb B\big]_{jh}<0$ by assumption and \emph{(i)}, we have $\big[\operatorname{Id} - \mathbb A^{-1}\mathbb B\big]_{ih}=\big[\mathbb A^{-1}(\mathbb A - \mathbb B)\big]_{ih}=\sum_{k\neq j} \mathbb A^{-1}_{ik}\big[\mathbb A - \mathbb B\big]_{kh} + \mathbb A^{-1}_{ij}\big[\mathbb A - \mathbb B\big]_{jh} < \sum_{k\neq j} \mathbb A^{-1}_{ik}{\mathbb A}_{kh} = \big[\mathbb A^{-1}\mathbb A\big]_{ih}=0$, where we have used that $\mathbb A_{jh}=\operatorname{Id}_{jh}=0$ for $x_j\in-I,\ j\neq h$.
\bigskip

We turn now to the original claim. Let $x_i\in J\big[\operatorname{Id} - \mathbb A^{-1}\mathbb B\big]^c$ such that there is a walk $\omega$ from $x_i$ to $J\big[\mathbb A - \mathbb B\big]$ in graph$\big(\mathbb A - \mathbb B\big)$. Assume without loss of generality that $\omega$ has no closed subwalks. We want to find a shorter walk to $J\big[\operatorname{Id} - \mathbb A^{-1}\mathbb B\big]$ in graph$\big(\operatorname{Id} - \mathbb A^{-1}\mathbb B\big)$. 

Due to \emph{(ii)}, it must be either $x_i\in\overline I$ or $x_i\in -I$, since states in $\overline I^c\cap (-I)^c$ correspond to SDD rows. In the first case, \emph{(iv)} implies that $\omega$ must have the form $\omega:x_i=y_0\to\dots \to y_n\to x_j\to x_h\to\dots$ for some $x_j\in (-I)\cap J\big[\mathbb A - \mathbb B\big]^c$, $y_m\in\overline I$ for all $m$ and $h\neq j,i$ (recall that $\omega$ has no closed subwalks). Then \emph{(v)} tells us that one can skip states and walk from $x_i$ to $x_h$ with a single step, through graph$\big(\operatorname{Id} - \mathbb A^{-1}\mathbb B\big)$. Afterwards, one can continue with the same path for states in $-I$ (if any), in light of \emph{(iii)}. The latter subwalk will either reach $J\big[\mathbb A - \mathbb B\big]$ (hence, $J\big[\operatorname{Id} - \mathbb A^{-1}\mathbb B\big]$ as per \emph{(ii)}) or go back to $\overline I$ and repeat the same procedure. Since there are finitely many states, $J\big[\mathbb A - \mathbb B\big]$ must eventually be reached in this way. 

The case of $x_i\in -I$ follows with the same reasoning. 
\end{proof}

%%%%%%%%%%%%%%%%%%%%%%%%%%%%%%%%%%%%%%%%%%%%%%% BIBLIOGRAPHY %%%%%%%%%%%%%%%%%%%%%%%%%%%%%%%%%%%%%%%%%%%%%%%%%%%%%%%%%%%%
\bibliographystyle{amsalpha}
\bibliography{references}
\addcontentsline{toc}{section}{References}

\end{document}